\documentclass[reqno]{amsart}

\usepackage{amssymb,amsfonts, amsmath, amsthm, graphicx, enumerate, upgreek}
\usepackage{mathtools, etoolbox, bm, mathdots,wrapfig}
\usepackage[table]{xcolor}
\usepackage{tikz}
\usepackage[all,arc]{xy}
\usepackage[utf8]{inputenc}
\usepackage[english]{babel}

\usepackage[pagebackref, hidelinks]{hyperref} 

\mathtoolsset{showonlyrefs, showmanualtags} 

\newtheorem{thm}{Theorem}[section]
\newtheorem*{thm*}{Theorem}

\newtheorem{cor}[thm]{Corollary}
\newtheorem*{cor*}{Corollary}
\newtheorem{prop}[thm]{Proposition}
\newtheorem*{prop*}{Proposition}
\newtheorem{lem}[thm]{Lemma}
\newtheorem*{lem*}{Lemma}

\newtheorem*{conj*}{Conjecture}

\newtheorem*{quest*}{Question}

\theoremstyle{definition}
\newtheorem{defn}[thm]{Definition}

\newtheorem{exmp}[thm]{Example}

\theoremstyle{remark}
\newtheorem{rem}[thm]{Remark}

\newtheorem*{ack*}{Acknowledgements}

\newcommand{\modsf}{\mathsf{mod}\ \!} 

\newcommand{\N}{\mathbb{N}} 
\newcommand{\Z}{\mathbb{Z}} 

\newcommand{\pdim}{{\rm projdim}\ \!} 
\newcommand{\idim}{{\rm injdim}\ \!} 
\newcommand{\gldim}{{\rm gldim}\ \!} 

\DeclareFontFamily{U}{wncy}{}
\DeclareFontShape{U}{wncy}{m}{n}{<->wncyr10}{}
\DeclareSymbolFont{mcy}{U}{wncy}{m}{n}
\DeclareMathSymbol{\Sh}{\mathord}{mcy}{"78} 
\DeclareMathOperator{\sh}{\Sh\,}

\newcommand{\id}{\mathrm{id}} 
\DeclareMathOperator{\Hom}{Hom}
\DeclareMathOperator{\Ext}{Ext}
\DeclareMathOperator{\Tor}{Tor}
\DeclareMathOperator{\rad}{rad}
\DeclareMathOperator{\wt}{wt}

\title[Finite generation for Gorenstein monomial algebras]{Finite generation for Hochschild cohomology\\ of Gorenstein monomial algebras} 
\author[Dotsenko]{Vladimir Dotsenko} 
\address{School of Mathematics, Trinity College, Dublin 2, Ireland}
\email{vdots@maths.tcd.ie}
\author[G\'elinas]{Vincent G\'elinas} 
\address{School of Mathematics, Trinity College, Dublin 2, Ireland}
\email{vgelinas@maths.tcd.ie}
\author[Tamaroff]{Pedro Tamaroff} 
\address{School of Mathematics, Trinity College, Dublin 2, Ireland}
\email{pedro@maths.tcd.ie}
\date{}

\dedicatory{To Ed Green with deep admiration of his work on the homology theory of monomial algebras}

\begin{document}

\begin{abstract} 
We show that a finite dimensional monomial algebra satisfies the finite generation conditions of Snashall--Solberg for Hochschild cohomology if and only if it is Gorenstein. This gives, in the case of monomial algebras, the converse to a theorem of Erdmann--Holloway--Snashall--Solberg--Taillefer. We also give a necessary and sufficient combinatorial criterion for finite generation.  
\end{abstract}

\maketitle 
\setcounter{tocdepth}{1}
\tableofcontents

\section*{Introduction} 

The present paper is intended as a contribution to the study of support varieties for finite dimensional algebras. The introduction of support varieties in modular representation theory \cite{Car83}, relying on the finite generation theorems of Evens~\cite{Evens}, Golod~\cite{Golod}, and Venkov~\cite{Venkov} for group cohomology, has revolutionised the subject since their appearance and has led to deep structural insight into the stable module category. 

A good theory of support varieties for finite dimensional algebras $\Lambda$ was developed by Snashall and Solberg \cite{SS04} under the hypothesis that certain finite generation (${\bf Fg}$) conditions hold for Hochschild cohomology. Determining which finite dimensional algebras $\Lambda$ satisfy ${\bf Fg}$ is an important open problem of representation theory. Furuya and Snashall \cite{FuSn} gave some explicit examples of Gorenstein, non self-injective monomial algebras satisfying ${\bf Fg}$. Furthermore, a theorem of Erdmann, Holloway, Snashall, Solberg, and Taillefer \cite[Thm. 1.5]{EHSST} states that any algebra satisfying ${\bf Fg}$ is necessarily Gorenstein. We establish the converse in the monomial case: 

\begin{thm*}[Th. \ref{gorensteinfg}] A monomial
algebra satisfies the conditions ${\bf Fg}$ if and only if it is Gorenstein. 
\end{thm*}

In \cite{Nagase}, Nagase proved that a Nakayama algebra satisfies the conditions ${\bf Fg}$ if and only if it is Gorenstein. This is a particular case of our result, since Nakayama algebras are monomial algebras for special type of quivers (cycles or type A quivers). Our methods are however very different, and rely on the recent work \cite{BG} of Briggs and the second named author who proposed, for a finite dimensional algebra~$\Lambda$, an approach towards studying the conditions ${\bf Fg}$ in terms of the canonical $A_\infty$-structure on the Yoneda algebra $\Ext_\Lambda^*(\Bbbk, \Bbbk)$ of the module $\Bbbk = 
\Lambda/\rad\Lambda$. Their work relies on earlier ideas of Green, Snashall, and Solberg \cite{GSS06} to study the ring structure on Hochschild cohomology in terms of its image under the characteristic homomorphism 
 \[
\varphi_\Bbbk: {\rm HH}^*(\Lambda, \Lambda) \to {\rm Ext}^*_\Lambda(\Bbbk, \Bbbk). 
 \]
We regard this as a useful concrete illustration of how one can apply the canonical $A_\infty$-algebra structure on the Yoneda algebra towards determining more classical homological invariants, such as Hochschild cohomology. This philosophy was advocated in \cite{LPWZ0,LPWZ}, but since higher structures of the Yoneda algebras are generally quite hard to compute, not many applications emerged so far. 

Knowledge of the canonical $A_\infty$-algebra structure of $\Ext^*_\Lambda(\Bbbk, \Bbbk)$ is equivalent, by Prout\'e's theorem \cite{Pro11}, to knowing the minimal quasi-free dg algebra resolution $\widetilde{\Lambda} \xrightarrow{\sim} \Lambda$, what is known as the minimal model of $\Lambda$. In the case of monomial algebras, this was done successfully by the third named author, who determined in \cite{Tamaroff} the minimal model of each monomial algebra $\Lambda$ and used it to compute the canonical $A_\infty$-algebra structure on $\Ext^*_\Lambda(\Bbbk, \Bbbk)$. This class of algebras is therefore a natural starting point for exploring the philosophy of~\cite{LPWZ0,LPWZ}. Moreover, as one can approximate general algebras by monomial ones by means of Gr\"obner bases, a thorough understanding of the monomial case should serve as the basis for a more general line of investigation. 

One of the main reasons that various invariants of monomial algebras can be computed explicitly is the extra grading that can be utilised. All important vector space and homology groups associated to a monomial algebra $\Lambda = kQ/I$ have a grading by the category $\mathfrak{C}(Q)$ freely generated by the quiver $Q$. This means that some generally structureless vector space acquire distinguished bases, and thus become naturally identified with their linear duals, leading to elegant formulas that are not available otherwise; in particular, the Tor groups  $\Tor^\Lambda_*(\Bbbk, \Bbbk)$ have combinatorial bases of the so called Anick chains. This was already noted and used in a recent preprint~\cite{Herscovich19}. For us, this observation leads to intricate vanishing patterns for higher structures of Yoneda algebras that can be regarded as analogues of the formulas of He and Lu for higher structures associated to $N$-Koszul algebras \cite{HL05}. Those vanishing patterns are at heart of some of our arguments, as they allow us to produce some explicit $A_\infty$-central elements of Yoneda algebras, arising from what we call ``stable relation cycles''. A slightly weaker notion of stability for relation cycles was introduced by Green, Snashall and Solberg in \cite{GSS06}; our approach to stability is directly motivated by applications to higher structures of the Yoneda algebra.

Our results also allow us to give a combinatorial characterisation of the Gorenstein property for monomial algebras in terms of Anick chains of sufficiently large length. This builds on recent work of Chen, Shen, and Zhou \cite{CSZ} who introduced the notion of perfect paths for monomial algebras in their classification of indecomposable Gorenstein-projective modules. Their work also shows that the minimal resolution of $\Bbbk$ over a Gorenstein monomial algebra is eventually periodic. A consequence of our work is the following reflection of that periodicity property on the level of Hochschild cohomology as follows. 

\begin{thm*}[Th. \ref{periodicHH}] Let $\Lambda$ be a Gorenstein monomial algebra, of Gorenstein dimension $d$. Then Hochschild cohomology is eventually periodic; more precisely, there exists an element $\chi \in {\rm HH}^*(\Lambda, \Lambda)$ of even degree $p$ such that taking cup product with $\chi$ gives an isomorphism 
\begin{align*}
\chi \smile - \colon {\rm HH}^n(\Lambda, \Lambda) \xrightarrow{\cong} {\rm HH}^{n+p}(\Lambda, \Lambda) \quad  \textup{ for all } n \geq d+1. 
\end{align*}
\end{thm*}

For a Gorenstein algebra $\Lambda$, we let $\widehat{ {\rm HH}^*}(\Lambda, \Lambda)$ denote its Tate--Hochschild cohomology. Our results imply the following elegant statement. 
\begin{thm*}[Cor. \ref{TateHH}] Let $\Lambda$ be a Gorenstein monomial algebra. Then Tate-Hochschild cohomology is given by periodic Hochschild cohomology:
\begin{align*}
		\widehat{ {\rm HH}^*}(\Lambda, \Lambda) \cong {\rm HH}^*(\Lambda, \Lambda)[\chi^{-1}].	
\end{align*} 
\end{thm*}

\subsection*{Structure of the paper} 

In Section \ref{sec:TorExt}, we discuss various special features of (co)homology monomial algebras and the related higher structures. In particular, we establish some vanishing patterns for higher structures of Yoneda algebras that are at heart of some of the main results of this paper. We also present a new perspective of some classical results, including the bimodule resolution of Bardzell, and a previously unpublished example of a monomial algebra with highly nontrivial higher structure in its Yoneda algebra. 
In Section \ref{sec:PerfGor}, we discuss the combinatorics of perfect cycles for monomial algebras due to Chen, Shen, and Zhou, and present an elegant combinatorial description of Anick chains for Gorenstein monomial algebras.   
In Section \ref{sec:Ainf}, we recall the {\bf Fg} conditions of Snashall--Solberg, and reinterpret them in terms of higher commutators and $A_\infty$-centres of $\Ext^*_\Lambda(\Bbbk, \Bbbk)$. We use the previously established vanishing patterns of higher structures to obtain a combinatorial criterion for $A_\infty$-centrality in the monomial case. 
In Section \ref{sec:Period}, we attach ``periodicity operators'' in ${\rm Ext}^*_\Lambda(\Bbbk, \Bbbk)$ to any stable relation cycle, prove that periodicity operators are $A_\infty$-central, and present the ring of periodicity operators explicitly.
In Section \ref{sec:main}, we prove the main result of this paper, namely that the monomial algebras satisfying the ${\bf Fg}$ conditions are exactly the Gorenstein ones.
In Section \ref{sec:Further}, we offer a combinatorial characterisation of the ${\bf Fg}$ conditions, and record some interesting results on the structure of Hochschild cohomology.

\subsection*{Conventions and notation} 
We use the language of quivers; a quiver is a directed graph $Q = (Q_0, Q_1, s, t)$ with the set of vertices $Q_0$ and the set of arrows $Q_1$, where for any arrow $\alpha \in Q_1$, we denote by $t(\alpha), s(\alpha) \in Q_0$ its target and source respectively. All quivers are assumed finite and connected. We use the notation $Q_m$ for the set of paths in $Q$ of length $m$, and $\mathcal{P}_Q = Q_0 \cup Q_1 \cup Q_2 \cup \cdots$ for the set of all paths in~$Q$. We call a path in $Q$ non-trivial if it has length at least one. 

All algebras in this article are quotients of the path algebras $kQ:=\mathop{\mathrm{span}}_k\mathcal{P}_Q$, where $k$ denotes the ground field. The product convention for paths in $kQ$ is given by function composition, so that $\alpha \beta = 0$ unless $s(\alpha) = t(\beta)$. We denote by $\mathfrak{m}$ the Jacobson radical $\rad kQ$; it is spanned by all non-trivial paths. Unless specified otherwise, we use the notation $\Lambda = kQ/I$ for a finite dimensional quotient of the path algebra with $I \subseteq \mathfrak{m}^2$ admissible (so that $\mathfrak{m}^n\subseteq I$ for some $n$; this is automatic in the case of finite dimensional monomial algebras). We let $\Bbbk := kQ_0$ denote the corresponding semisimple $k$-algebra, and use the notation $\mathfrak{r} = \rad\Lambda$ for the Jacobson radical of $\Lambda$. This allows us treat $\Lambda$ as an augmented algebra over $\Bbbk$ with augmentation $\varepsilon: \Lambda \twoheadrightarrow \Lambda/\mathfrak{r} = \Bbbk$. All unadorned tensors are over $\Bbbk$. 

All modules over finite dimensional algebras are taken to be right modules and finitely generated; we simply call them finite. Elements of graded algebras and modules are always taken to be homogeneous. We refer to the internal grading of algebras and modules as the weight grading. The (co)homology groups of graded algebras are naturally bi-graded; one grading is given by the weight, and the other by (co)homological degree, which we simply refer to as degree. The weight and degree of an element $x$ are denoted by $\wt(x)$ and $|x|$ respectively. 

Unless otherwise specified, we focus on the case of monomial algebras with ideal of relations generated by a set of paths which we assume minimal with respect to inclusion of paths.

When dealing with some formulas for homology, we find it useful to utilise the standard convention for homological suspension, implementing it via a formal symbol $s$ of homological degree~$1$.

\subsection*{Acknowledgments} The second named author is supported by Simons Foundation (through a postdoctoral fellowship at Hamilton Mathematics Institute). We are grateful to Ben Briggs for useful discussions. A crucial thank you is due to Joe Chuang and Alastair King who kindly shared with us their unpublished manuscript from fifteen years ago that highlights the importance of previous work of Gruenberg \cite{Gruenberg68} for computing higher structures on Ext algebras of monomial algebras; the beautiful example of Section \ref{sec:ExampleCK} is also coming from their unpublished note (and is reproduced here with their permission). Their work on higher structures ultimately led to a paper on functorial non-minimal resolutions of general associative algebras \cite{CK}, and their ideas related to the specific case of monomial algebras never made it to that paper. However, it was of utmost importance for the genesis of this paper, and we cannot thank them enough for sharing their work with us. 

\section{\texorpdfstring{$\Tor$}{Tor} and \texorpdfstring{$\Ext$}{Ext} of monomial algebras, and their higher structures}\label{sec:TorExt}

\subsection{Two formulas for Tor groups}  

In general, for an augmented algebra $\Lambda$, the Tor groups $\Tor_\bullet^\Lambda(\Bbbk,\Bbbk)$ of the trivial module can be computed as the homology of the bar construction $B_\bullet(\Lambda)=\{ (s\mathfrak{r})^{\otimes n} , d \}$. This formula is useful theoretically but in practice has a lot of limitations: the underlying space of bar construction is too big. In this section, we shall discuss two formulas that produce the Tor groups on the nose; one is valid for any augmented algebra $\Lambda$ presented by generators and relations, and the other only makes sense for monomial algebras. A few of the results of this paper arise from comparing those two formulas. To the best of our knowledge, this has not been done explicitly before; the only sources we know where the two formulas appear at the same time are the article \cite{BK} by Butler and King, and the survey of Ufnarovski~\cite{Ufn95}. 

The first formula for the $\Tor$ groups is as follows.
\begin{thm}\label{th:Govorov}
Let $\Lambda = kQ/I$ with $I \subseteq \mathfrak{m}^2$. 
For $p\in\mathbb N$ we have that
\begin{gather}
\Tor_{2p}^\Lambda(\Bbbk,\Bbbk)\cong \dfrac{I^p\cap \mathfrak{m} I^{p-1}\mathfrak{m}}{I^p\mathfrak{m}+\mathfrak{m} I^p}, \label{eq:Gov1}\\
\Tor_{2p+1}^\Lambda(\Bbbk,\Bbbk)\cong \dfrac{I^p\mathfrak{m}\cap \mathfrak{m}I^p}{I^{p+1}+\mathfrak{m} I^p\mathfrak{m}}. \label{eq:Gov2}
\end{gather}
\end{thm}
Similar formulas were first discovered in the case of integral homology of groups presented by generators and relations by Gruenberg~\cite{Gruenberg68}, and then established in the case of associative algebras over a field by Govorov~\cite{Govorov73}. The work of Govorov remained largely unnoticed, and it seems that the first prominent appearance of this result for quiver algebras is in a paper of Bongartz~\cite{Bongartz83} who refers to a private communication from Butler inspired by work of Gruenberg. We shall refer to these formulas for $\Tor$-groups as the Govorov--Gruenberg formulas.    

The Govorov--Gruenberg formulas are way too general to be of computational relevance for arbitrary algebras. However, in the case of monomial algebras they can be made completely combinatorial, if we note that all the vector spaces in those formulas have compatible combinatorial bases, e.g. the vector space $\mathfrak{m}I^p$ has a basis of paths obtained as concatenations of a nontrivial path with a path containing at least $p$ disjoint occurrences of the monomial relations. Once those compatible bases are identified, all intersections and quotients are computed on the level of combinatorial bases by taking the intersections and set-theoretic differences respectively.  

The second formula for $\Tor$-groups which is only valid for algebras with monomial relations utilises combinatorics of the so called \emph{Anick chains}. We shall use that terminology, even though it is inaccurate historically; in the case of algebras with monomial relations that combinatorics seems to have been first discovered by Green, Happel and Zacharia~\cite{GHZ85}, but is better known from works of Anick \cite{Anick86} and Anick and Green \cite{AG87} where the case of monomial algebras was upgraded to the case of augmented algebras with a Gr\"obner basis of relations. 

Let us recall the definition of the set $C_n$ of (right) Anick $n$-chains and of a tail of an Anick chain; this is done by induction on $n\ge 0$. We let $C_0 = Q_1$ be the set of arrows, and any $a \in C_0$ is its own tail. Suppose that $C_{n-1}$ is defined, as well as tails of $(n-1)$-chains. An $n$-chain is a path $\upgamma$ in $\mathcal{P}_Q$ such that
\begin{enumerate}[(i)]
\item we can write $\upgamma = \upgamma't$ with $\upgamma' \in C_{n-1}$,
\item if $t'$ is the tail of $\upgamma'$, then $t't$ has a right divisor which is a monomial relation, 
\item no proper left divisor of $\upgamma$ satisfies both (i) and (ii). 
\end{enumerate}
By definition, the tail of $\upgamma$ is the path $t_\upgamma := t$. Note that $C_1$ is the set of the defining relations of $\Lambda$, with the tail of each monomial relation given by the path obtained after removing the leftmost arrow in it. It is shown in \cite{Anick86,AG87} that a path $\upgamma$ admits at most one structure of a chain: if $\upgamma$ is an $n$-chain with tail $t_\upgamma$, then $n$ and $t_\upgamma$ are uniquely determined. 

The following result is essentially due to Green, Happel, and Zacharia~\cite{GHZ85}.
\begin{thm}\label{th:AnickTor}
For each $n\in\mathbb N$ we have isomorphisms
 \[
\Tor_n^\Lambda(\Bbbk,\Bbbk)\cong kC_{n-1} .
 \]
\end{thm}

To prove this result, one uses the combinatorial definition above to prove a slightly stronger statement: the minimal free right module resolution of $\Bbbk$ over $\Lambda$ can be proved to have the form 
\begin{align}
\cdots \to	kC_n \otimes_\Bbbk \Lambda \xrightarrow{d} kC_{n-1} \otimes_\Bbbk \Lambda \xrightarrow{d} \cdots \xrightarrow{d} k C_{1} \otimes_\Bbbk \Lambda \xrightarrow{d} kC_0 \otimes_\Bbbk \Lambda \to \Lambda \to 0 
\end{align}
with $d(\upgamma \otimes 1) = \upgamma'\otimes t_\upgamma$. (Of course applying the functor $-\otimes_\Lambda\Bbbk$ to this resolution, we immediately see that $\Tor_n^\Lambda(\Bbbk,\Bbbk)\cong kC_{n-1}$, as required.) 

Similarly, one can define the set of left Anick chains $C'_n$ by considering \emph{heads} instead of tails, and thus obtain the minimal free left module resolution of $\Bbbk$ over $\Lambda$. Bardzell \cite{Bardzell97} proved that $C'_n = C_n$ for all $n \geq 0$, and so for any $n$-chain $\upgamma$, we can write $\upgamma = \upgamma't_\upgamma = h_\upgamma \upgamma''$ for unique $(n-1)$-chains $\upgamma', \upgamma''$, with tail $t_\upgamma$ and head $h_\upgamma$. Another explanation of this symmetry from the point of view of the Govorov--Gruenberg formulas was given by Butler and King, see \cite[Th.~8.2, Remark 8.3]{BK}.

At this point, it is appropriate to make a remark on how the extra grading mentioned earlier can be utilised. Since the path algebra $kQ$, as well as its ideals generated by monomials, is graded by the category $\mathfrak{C}(Q)$ generated by $Q$, that grading descends to all monomial quotients of $kQ$, and to bar constructions of those quotients, where the boundary map manifestly respects the grading. Thus, every group $\Tor_n^\Lambda(\Bbbk,\Bbbk)$ is graded by $\mathfrak{C}(Q)$. From the description of Tor via Anick chains, it is clear that for every path $p\in\mathcal{P}_Q$, the $p$-graded component of Tor is of dimension at most one (and when it is of dimension one, the corresponding one-dimensional space has the Anick chain $p$ as its basis). Since the Tor groups (together with their $\mathfrak{C}(Q)$-gradings) are well defined, all the descriptions we outlined (via the bar construction, of Govorov and Gruenberg, and of Green--Happel--Zacharia) must give the same result. The first immediate consequence of that is the result of Bardzell on the coincidence of left and right Anick chains mentioned in the previous paragraph: both of those index nonzero graded components in the homology of the bar construction. Another interesting observation is that the above property of Tor stating that for each path $p\in\mathcal{P}_Q$, the $p$-graded component of Tor is of dimension at most one is precisely the key result of the paper of Iyudu \cite{Iyudu} who suggests that it was conjectured by Kontsevich in 2015. 

\subsection{Higher structures of Tor and Ext}

Recall that an $A_\infty$-algebra over $\Bbbk$ is a graded $\Bbbk$-bimodule $E$ equipped with multilinear operations $m_n\colon E^{\otimes n} \to E$ of degree $2-n$ satisfying the Stasheff identity 
\begin{align}
	\sum_{r+s+t=n} (-1)^{rs+t} m_{r+1+t} \circ ({\rm id}^{\otimes r} \otimes m_s \otimes {\rm id}^{\otimes t}) = 0
\end{align}
for each $n\ge1$. These identities imply that $m_1$ is a differential for $E$ and that $m_2$ is associative up to coboundary given by $m_3$, and we call $(E, \{ m_n\})$ minimal when $m_1 = 0$. In this case, $(E, m_2)$ is an honest graded augmented $\Bbbk$-algebra with some extra data given by the higher products $\{ m_n \}_{n \geq 3}$. We refer the reader to \cite[Chapter 9]{LV12} for further information on $A_\infty$-algebras. 

The $A_\infty$-algebra $E$ is said to be strictly unital if it contains an element $1$ for which $$m_2(1, x) = x = m_2(x, 1) \text{ and }  m_{n}(\dots, 1, \dots) = 0 \text{ for } n \geq 3. $$ The $A_\infty$-algebra $E$ is augmented over $\Bbbk$ if it is equipped with degree zero morphisms $\varepsilon \colon E \leftrightarrows \Bbbk \colon \eta$ satisfying $\varepsilon \eta =\id$, so that it is strictly unital for $1:=\eta(1)$, and the products preserve the augmentation ideal $\overline{E} := {\rm ker}(\varepsilon)$. We always take $E$ to be augmented. 

In particular, the Yoneda algebra $E = \Ext^*_\Lambda(\Bbbk, \Bbbk)$ of any augmented dg $\Bbbk$-algebra admits a minimal augmented $A_\infty$-algebra structure making it $A_\infty$-equivalent to the dg algebra ${\rm RHom}_\Lambda(\Bbbk, \Bbbk)$. Such a structure is unique up to $A_\infty$-isomorphism; it recovers $\Lambda$ up to an $A_\infty$-quasi-isomorphism. For the rest of this article $\Ext^*_\Lambda(\Bbbk, \Bbbk)$ will always be considered equipped with such an $A_\infty$-structure, which we shall call canonical. (Dually, there is a notion of an $A_\infty$-coalgebra; for each augmented $\Bbbk$-algebra $\Lambda$, the graded space $\Tor_*^\Lambda(\Bbbk,\Bbbk)$ has an $A_\infty$-coalgebra structure which is $A_\infty$-equivalent to the bar construction $B_*(\Lambda)$, regarded as a dg coalgebra with the deconcatenation coproduct.)

For a monomial algebra $\Lambda$, the canonical $A_\infty$-structure on $\Ext$ was determined by the third named author in \cite{Tamaroff}; this result generalises the classical formula of Green and Zacharia \cite{GZ94} for the Yoneda product. By Theorem~\ref{th:AnickTor}, if we denote by $(-)^{\vee}$ the $\Bbbk$-dual, we have identifications $${\rm Ext}^*_\Lambda(\Bbbk, \Bbbk) = (\Bbbk \oplus kC_{*-1})^{\vee} = \Bbbk \oplus kC_{*-1}^\vee,$$ where we set $C_{-1} = \emptyset$ and moreover abuse notation and let $C_r^\vee \subseteq \mathcal{P}_{Q^{op}}$ consist of the paths in $C_r$ with opposite orientation. More generally we identify the sets of paths $\mathcal{P}_Q$ and $\mathcal{P}_{Q^{op}}$ via the notation $p \leftrightarrow p^{\vee}$ which reverses orientation, so that $(pq)^\vee = q^\vee p^\vee$.  

Consider the Yoneda algebra $E = \Ext^*_\Lambda(\Bbbk, \Bbbk) = \Bbbk \oplus kC_{*-1}^{\vee}$. For each $n \geq 2$, let us define morphisms $m_n\colon E^{\otimes n} \to E$ on (duals of) Anick chains $\upgamma_i \in C_{r_i}^\vee$ by 
\begin{align}\label{higherproducts}
		m_n(\upgamma_1, \dots, \upgamma_n) = \begin{cases} (-1)^N \upgamma_1 \dots \upgamma_n & \textup{ if } \upgamma_1 \dots \upgamma_n \in C^\vee_{r_1 + \dots + r_n + 1}, \\ 
		0 & \textup{ else,}\end{cases} 
\end{align}
where $N = \sum_{i<j} r_i(r_j+1) + r_1 + \sum_{i} r_i$. 
	
\begin{thm}\label{TamaroffThm} 
These maps model the canonical $A_\infty$-algebra structure on ${\rm Ext}^*_\Lambda(\Bbbk, \Bbbk)$. 
\end{thm}

\begin{proof}
This is analogous to \cite[Th.~4.2]{Tamaroff} but uses left Anick chains instead of right ones; one just has to remark that in this case the only non-vanishing tree appearing in the homotopy transfer formula for the higher products is the left comb, and that no extra signs arise, as opposed to what happens for the right comb.
\end{proof}

\subsection{Example of non-vanishing higher products}\label{sec:ExampleCK}

At this point we already have enough information to present an interesting example due to Chuang and King of a very simple monomial quiver algebra which however has highly nontrivial higher structure on its Yoneda algebra. The only available source in the literature where an example of a similar nature appears is a paper of Conner and Goetz \cite{CG}; however, in that case the authors show that \emph{a certain} canonical $A_\infty$-structure is highly nontrivial, and do not establish that it holds for \emph{every} canonical $A_\infty$-structure. Hence we believe that this example should be of interest to experts. 

Let us consider the cyclic quiver $C_5$ with the arrows $a,b,c,d,e$ in the cyclic order.
\begin{wrapfigure}{r}{0.27\textwidth}
\vspace{-1 em}
\begin{tikzpicture}[auto, scale = 0.6]
    \foreach \a in {1,2,3,4,5}
    {
        \node (u\a) at ({\a*72}:2){$\bullet$};
        \draw [latex-,
        		line width = 1.15 pt,
        		domain=\a*72-65:\a*72-10] plot ({2*cos(\x)}, {2*sin(\x)});}
    \node (a1) at ({1*72+35}:2.5){$e$};
    \node (a2) at ({2*72+35}:2.5){$a$};
    \node (a3) at ({3*72+35}:2.5){$b$};
    \node (a4) at ({4*72+35}:2.5){$c$};
    \node (a5) at ({5*72+35}:2.5){$d$};  
\end{tikzpicture}
\vspace{-2 em}
\end{wrapfigure}
The algebra $CK_5$ is defined by generators and relations as $kC_5/(abcd,bcde,deab)$, which is a finite-dimensional algebra. We will prove that a canonical $A_\infty$-structure on its Yoneda Ext-algebra cannot be ``too simple''. 

\begin{thm}
The Yoneda algebra $B=\Ext^*_{CK_5}(\Bbbk,\Bbbk)$ has infinitely many non-vanishing higher products, for any canonical $A_\infty$-structure. 
\end{thm}

\begin{proof}
The Anick $1$-chains of the algebra $CK_5$ are the relations $abcd$, $bcde$, and $deab$, and the Anick $2$-chains are $abcde$, $bcdeab$, $deabcd$.
Furthermore, it turns out that for each $n\ge 3$ there are just two Anick chains: for $n=2s-1$, the monomials

 \[
bcde(abcde)^{s-2}abcd \quad \text{and} \quad de(abcde)^{s-1}ab,
\]

and for $n=2s$, the monomials

\[
bcde(abcde)^{s-1}ab \quad \text{and} \quad de(abcde)^{s-1}abcd .
\]

This can be easily proved by induction on $n$. We identify these chains with their duals in the Yoneda algebra~$B$.

We prove the statement of the theorem in a very concrete way. Namely, we demonstrate that for the element $$\omega_n=d\otimes e\otimes (abcde)^{\otimes (n-4)}\otimes a \otimes b\in B^{\otimes n},$$ and for any canonical $A_\infty$-structure $\{\nu_n\}$ on $B$, we have $\nu_n(\omega_n)\ne 0$.

First, we note that the total degree of $\omega_n$ is $1+1+3(n-4)+1+1=3n-8$, and the concatenation $de(abcde)^{n-4}ab$ is a chain of length $2(n-3)-1$, so an element of homological degree $2(n-3)=2n-6$. Since $2n-6=3n-8+2-n$, the homological degrees match, so for the canonical structure $\{\mu_n\}$ of Tamaroff, we have $\mu_n(\omega_n)\ne 0$.

Suppose that $\nu_n$ is another canonical structure, and $\{f_n\}$ is an $\infty$-equivalence between them; without loss of generality, $f_1=\id$. Let us denote by $\Phi_n$ the set of all $n$-fold products of chains $\upgamma_1\otimes \cdots\otimes \upgamma_n$ in $B^{\otimes n}$ with the following properties:

\begin{itemize}

\item[(i)] all degrees $|\upgamma_i|$ are odd,

\item[(ii)]  if $|\upgamma_i|=1$, then $i\in\{1,2,n-1,n\}$,

\item[(iii)] if $|\upgamma_1|=|\upgamma_2|=1$, then $\upgamma_1=d$, $\upgamma_2=e$,

\item[(iv)] if $|\upgamma_{n-1}|=|\upgamma_n|=1$, then $\upgamma_{n-1}=a$, $\upgamma_n=b$,

\item[(v)] $|\upgamma_i|>1$ for some $i$.

\end{itemize}

By direct inspection, every connected $r$-fold subtensor $\psi$ of $\phi\in\Phi_n$ is in $\Phi_r\setminus\{\omega_r\}$. Also, for the canonical structure of Tamaroff, if $\mu_n(\phi)\ne 0$ for some $\phi\in\Phi_n$, then $\phi=\omega_n$, which easily follows from the observation that concatenations of pairs of elements of degree greater than $2$ cannot be found as divisors of any Anick chain.

Using these observations, we prove by induction on $n$ that for any minimal canonical $A_\infty$-structure $\{\nu_m\}$ and for any $\phi\in\Phi_n$ we have $\nu_n(\phi)=\mu_n(\phi)$. The morphism identity states that

 \[
\sum_{i+j+k=n}\pm f_{i+1+k}(1^{\otimes i}\otimes\mu_j\otimes 1^{\otimes k})(\phi)=
\sum_{i_1+\ldots+i_r=n}\pm \nu_r(f_{i_1}\otimes\cdots\otimes f_{i_r})(\phi).
 \]

By induction and above observations, the only element on the left that does not vanish is $f_1(\mu_n)$, and the only element on the right that does not vanish is $\nu_n(f_1^{\otimes n})$ (since for $r<n$, in order for $\nu_r$ to not vanish, $(f_{i_1}\otimes\cdots\otimes f_{i_r})(\phi)$ must create $\omega_n$ with a nonzero coefficient, which is impossible for degree reasons), so we are done.
\end{proof}

\subsection{Vanishing patterns for higher products}\label{sectionvanishingpatterns} 

Let us now demonstrate how one can combine the two formulas for Tor groups to prove some nontrivial results. 

\begin{prop}\label{vanishingpatterns} Consider the canonical $A_\infty$-algebra structure on $\Ext^*_\Lambda(\Bbbk, \Bbbk)$ from Theorem~\ref{TamaroffThm}. Suppose $\upgamma_1, \dots, \upgamma_n$ are elements of the $\Ext$ algebra. 
\begin{enumerate}[(i)] 
\item If at least three of them are of even degree, then $m_n(\upgamma_1, \dots, \upgamma_n) = 0$. 
\item If exactly two of them are of even degree, then $m_n(\upgamma_1, \dots, \upgamma_n) = 0$ unless those elements are $\upgamma_1$ and $\upgamma_n$. 
\item If exactly one element is of even degree, then $m_n(\upgamma_1, \dots, \upgamma_n) = 0$ unless that element is $\upgamma_1$ or $\upgamma_n$. 
\end{enumerate}
\end{prop}

\begin{proof}
By linearity, it is enough to prove this result in the case of all those elements being (duals of) Anick chains. Suppose that among them, we have $k$elements of even degrees $2d_1$, \ldots, $2d_k$, and $\ell$ elements of odd degrees $2e_1+1$, \ldots, $2e_\ell+1$; of course, $n=k+\ell$. We once again use the fact that Tor groups and their combinatorial bases are well defined, and identify those elements with the elements of the Govorov--Gruenberg spaces; from Formulas~\eqref{eq:Gov1} and~\eqref{eq:Gov2} it follows in particular that those elements are represented by paths in $I^{d_1}$, \ldots, $I^{d_k}$, $I^{e_1}\mathfrak{m}\cap \mathfrak{m} I^{e_1}$, \ldots, $I^{e_\ell}\mathfrak{m}\cap \mathfrak{m} I^{e_\ell}$ respectively. By Formula \eqref{higherproducts}, the value of the operation~$\mu_n$ on these elements is proportional to the concatenation of the corresponding paths. 

Note that the concatenation of paths from $I^{d_1}$, \ldots, $I^{d_k}$, $I^{e_1}\mathfrak{m}\cap \mathfrak{m}I^{e_1}$, \ldots, $I^{e_\ell}\mathfrak{m}\cap \mathfrak{m} I^{e_\ell}$ (in any order) is in $I^{d_1+\cdots+d_k+e_1+\cdots+e_\ell}$. 
At the same time, the result of applying $\mu_n$ to such words has homological degree
 \[
2(d_1+\cdots+d_k+e_1+\cdots+e_\ell)+\ell+2-n=2(d_1+\cdots+d_k+e_1+\cdots+e_\ell)+2-k .
 \]
In particular, for $k>2$, the degree of $\mu_n(\upgamma_1,\ldots,\upgamma_n)$ is less than $$d=2(d_1+\cdots+d_k+e_1+\cdots+e_\ell).$$ Applying Formulas~\eqref{eq:Gov1} and~\eqref{eq:Gov2} again, we conclude that the concatenation of all our elements is in the zero coset of the corresponding Govorov--Gruenberg space. 

Similarly, for $k=2$, the degree of $\mu_n(\upgamma_1,\ldots,\upgamma_n)$ is the even number $$2(d_1+\cdots+d_k+e_1+\cdots+e_\ell).$$ If we assume that at least one of the two elements $\upgamma_1$, $\upgamma_n$ is of odd degree, we observe that the concatenation of all our elements is in $\mathfrak{m} I^d+I^d\mathfrak{m}$, which is in the zero coset of the corresponding Govorov--Gruenberg space.

Finally, for $k=1$, the degree of $\mu_n(\upgamma_1,\ldots,\upgamma_n)$ is the odd number $$2(d_1+\cdots+d_k+e_1+\cdots+e_\ell)+1.$$ If we assume that both elements $\upgamma_1$, $\upgamma_n$ are of odd degree, we observe that the concatenation of all our elements is in $\mathfrak{m} I^d\mathfrak{m}$, which is in the zero coset of the corresponding Govorov--Gruenberg space.
\end{proof}

As a first consequence, one may simplify the signs arising in the higher products, at least to some extent.

\begin{cor}\label{newhigherproducts} Formula (\ref{higherproducts}) can be simplified to
\begin{align}
m_n(\upgamma_1, \dots, \upgamma_n) = \begin{cases} (-1)^{nr_1 + r_1r_n + r_1 + r_n} \upgamma_1 \dots \upgamma_n & \textup{ if } \upgamma_1 \dots \upgamma_n \in C^\vee_{r_1 + \dots + r_n + 1}, \\ 
0 & \textup{ else.}\end{cases} 
\end{align}
\end{cor}

\subsection{Application: Bardzell's resolution} 

In \cite{Bardzell97}, Bardzell constructed explicitly the minimal bimodule resolution of the diagonal bimodule over any algebra $\Lambda$ with monomial relations. By an intricate study of combinatorics of Anick chains, he established that the shape of the differential of such resolution alternates, depending on the homological degree. Let us explain how this phenomenon is an immediate consequence of Proposition \ref{vanishingpatterns}. 

Repeating the argument of \cite[Th.~4.2]{Herscovich18} \emph{mutatis mutandis} for the case of quiver algebras, we see that the canonical $A_\infty$-algebra structure on $\Ext^*_\Lambda(\Bbbk,\Bbbk)$ gives rise to a free bimodule resolution of the diagonal bimodule. The corresponding twisting cochain $\tau\colon\Tor_*^\Lambda(\Bbbk,\Bbbk)\to \Lambda$ is extremely easy to describe for a monomial quiver algebra $\Lambda=kQ/I$: it annihilates elements of homological degree different from one, and on elements of degree one it is given by the identity map under the identification $\Tor_1^\Lambda(\Bbbk,\Bbbk)= kQ_1\subset \Lambda$.

Now, the only thing that remains is to examine carefully the dual formulas of \ref{higherproducts}. We shall use the variant of that formula which describes the boundary map in the minimal model of $\Lambda$, and can be chosen in the form similar to \cite[Th.~4.1]{Tamaroff}, but with a different choice of signs arising from considering left Anick chains: 
\begin{equation}\label{diffmodel}
b(s^{-1}\upgamma)=\sum_{n\ge 2} (-1)^{|s^{-1}\upgamma_1|} s^{-1}\upgamma_1\otimes\cdots\otimes s^{-1}\upgamma_n ,
\end{equation}
 where the sum ranges through all decompositions of an Anick chain $\upgamma$. This shape of the formula is particularly convenient for computing the differential of the bimodule resolution of the diagonal module, since the twisting cochain $\tau$ is precisely the operator $\upgamma\mapsto s^{-1}\upgamma$ on $\upgamma$ of degree $1$. Thus, to obtain the answer, we need to classify all possible decompositions 
 \[
\upgamma=\upgamma_1\cdots \upgamma_n
 \]
with $|s^{-1}\upgamma|=|s^{-1}\upgamma_1|+\cdots+|s^{-1}\upgamma_n|+1$ and where in addition all $s^{-1}\upgamma_i$ but one are of homological degree~$0$. By Proposition \ref{vanishingpatterns}, this depends on the parity of~$|\upgamma|$. If $|\upgamma|$ is odd, then for the only element $s^{-1}\upgamma_i$ of positive homological degree, $|\upgamma_i|$ is even, and so $i=1$ or $i=n$. On the other hand, if $|\upgamma|$ is even, then for the only element $s^{-1}\upgamma_i$ of positive homological degree, $|\upgamma_i|$ is odd, so there are no constraints on the position of this element. Thus, Formula \ref{diffmodel} suggests that the image of the generator $e_{t(\upgamma)}\otimes \upgamma\otimes e_{s(\upgamma)}$ of the bimodule resolution $\Lambda\otimes kC_*\otimes\Lambda$ under the differential of that resolution is given by the formula
 \[
\left\{
\begin{aligned}
h_\upgamma\otimes \upgamma''\otimes e_{s(\upgamma)} - e_{t(\upgamma)}\otimes \upgamma'\otimes t_\upgamma, \quad &|\upgamma|\equiv 1\pmod{2},\\
\sum\alpha\otimes \tilde{\upgamma}\otimes \beta, \qquad &|\upgamma|\equiv 0\pmod{2},
\end{aligned}
\right.
 \]
where $\upgamma=h_\upgamma\upgamma''=\upgamma't_\upgamma$ are the decompositions of $c$ as the left chain and the right chain that factor out the head and the tail, and the sum is over all decompositions $c=\alpha\tilde{\upgamma}\beta$ with $\tilde{\upgamma}$ a chain of homological degree one less than $\upgamma$. This is precisely the result of \cite[Th.~4.1]{Bardzell97}, the main theorem of that paper.

\section{Perfect paths and Gorenstein algebras}\label{sec:PerfGor}

\subsection{Gorenstein-projective modules and related combinatorics}

Recall that a module $M \in \modsf \Lambda$ is called Gorenstein-projective if it satisfies the following conditions: 
\begin{enumerate}[i)]
\item There exists an acyclic complex of finite projectives
	\begin{align} 
		C_*: 	\cdots \to C_3 \xrightarrow{d} C_2 \xrightarrow{d} C_{1} \xrightarrow{d} C_0 \xrightarrow{d} C_{-1} \xrightarrow{d} C_{-2} \to \cdots 
	\end{align}
	such that ${\rm coker}(C_1 \xrightarrow{d} C_0) \cong M$. 
\item The dual complex ${\rm Hom}_\Lambda(C_*, \Lambda)$ is also acyclic. 
	\end{enumerate}

In this case the non-negative truncation of $C_*$ resolves $M$, and we call $C_*$ a complete resolution of $M$. Equivalently \cite[Chp. 4]{Bu}, $M$ is Gorenstein-projective if and only if 
	\begin{align}
		{\rm Ext}^i_\Lambda(M, \Lambda) = 0 \textup{ and } {\rm Ext}^i_{\Lambda^{op}}(M^*, \Lambda^{op}) = 0 \textup{ for all } i > 0	
	\end{align}
	and $M \cong M^{**}$ is reflexive, where $M^* = {\rm Hom}_\Lambda(M, \Lambda)$ is the dual module.

In \cite{CSZ}, Chen--Shen--Zhou classified the indecomposable Gorenstein-projectives over monomial algebras. Let us recall the combinatorial notions that underpin that classification. Note that for a monomial algebra $\Lambda$, the coset of a path $p \in \mathcal{P}_Q$ is zero in $\Lambda$ if and only if it is divisible by a relation. We let 
\begin{align} 
\mathcal{P}_\Lambda = \{ p \in \mathcal{P}_Q \ | \ p \textup{ is not divisible by any relation} \} 
\end{align} 
be the subset of nonzero paths in $\Lambda$, and the canonical surjection $kQ \twoheadrightarrow \Lambda$ identifies $\mathcal{P}_\Lambda$ with a basis for $\Lambda$. 

Given any path $p \in \mathcal{P}_\Lambda$, we define minimal left zero cofactors of $p$ to be those non-trivial paths $q \in \mathcal{P}_\Lambda$ with $s(q) = t(p)$, $qp = 0$, and for which no proper end segment $q'$ has the same property: if $q=q''q'$, and $q'p=0$, then $q=q'$. Similarly one defines minimal right zero cofactors. We let 
\begin{align}
	L(p) &= \{q \in \mathcal{P}_\Lambda \ | \ q \textup{ is a minimal left zero cofactors of } p \}\\ 
	R(p) &= \{q \in \mathcal{P}_\Lambda \ | \ q \textup{ is a minimal right zero cofactors of } p \}.
\end{align} 

\begin{defn} 
A pair of non-trivial paths $(p,q)$ in $\Lambda$ with $s(p) = t(q)$ is a perfect pair if $R(p) = \{ q \}$ and $L(q) = \{ p \}$. 
\end{defn}

\begin{defn} 
Let ${\bf p} = (p_0, p_1, \dots, p_{r-1})$ be a sequence of non-trivial paths in $\Lambda$ such that $s(p_{i}) = t(p_{i+1})$ for all $i \in \Z/r\Z$. We say that ${\bf p}$ is a perfect cycle if all pairs $(p_{i}, p_{i+1})$ are perfect; in such case we call the paths $p_i$ perfect as well. If $r$ is the minimal integer with this property, we call $r$ the period of ${\bf p}$. 
\end{defn}
Given a perfect cycle ${\bf p} = (p_0, p_1, \dots, p_{r-1})$, one may extend it by periodicity to ${\bf p} = (p_i)_{i \in \Z}$ by setting $p_{i+r} = p_i$. We then attach an unbounded, periodic complex of finite right projectives $Q_*^{({\bf p})}$, with terms $Q_i^{({\bf p})} = e_{s(p_i)} \Lambda $ and differentials given by left multiplication by the $p_i$: 
\begin{align} 
Q_*^{({ \bf p})}: \cdots \to e_{s(p_{i+3})}\Lambda  \xrightarrow{l_{p_{i+3}}} e_{s(p_{i+2})}\Lambda  \xrightarrow{l_{p_{i+2}}} e_{s(p_{i+1})}\Lambda  \xrightarrow{l_{p_{i+1}}} e_{s(p_{i})}\Lambda  \to \cdots 
\end{align} 
That this is a complex follows from the zero cofactor relation $p_i p_{i+1} = 0$ in $\Lambda$. Moreover, $Q_*^{({\bf p})}$ is acyclic as each pair $(p_i, p_{i+1})$ is perfect \cite[Prop.~4.4]{CSZ}. 

 Let $\mathcal{GP}(\Lambda)$ be the set of isomorphism classes of non-projective, indecomposable Gorenstein-projective right modules over $\Lambda$. 

\begin{thm}[{\cite[Th. 4.1]{CSZ}}]\label{GPclassification} Let $\Lambda$ be a monomial algebra. Then the map $p \mapsto p \Lambda$ is a bijection 
	\begin{align*}
		\{ \textup{perfect paths in } \mathcal{P}_\Lambda \} \leftrightarrow \mathcal{GP}(\Lambda). 
	\end{align*} 
	Moreover, given a perfect path $p = p_0$ in a perfect cycle ${\bf p} = (p_0, p_1, \dots, p_{r-1})$, the complex $Q_*^{({\bf p})}$ is a complete resolution of $p_0\Lambda$, and in particular its non-negative truncation 
	\begin{align}
		\cdots \to e_{s(p_{3})}\Lambda  \xrightarrow{l_{p_{3}}} e_{s(p_{2})}\Lambda  \xrightarrow{l_{p_{2}}} e_{s(p_{1})}\Lambda  \xrightarrow{l_{p_{1}}} e_{s(p_{0})}\Lambda  \xrightarrow{l_{p_0}} p_0 \Lambda \to 0
	\end{align} 
	is a minimal projective resolution of $p\Lambda = p_0 \Lambda$. 
\end{thm}

A monomial algebra $\Lambda$ may contain no perfect path and so may admit no non-trivial Gorenstein-projective module at all. At the opposite end of the spectrum, Gorenstein monomial algebras (of infinite global dimension) have an abundance of non-trivial Gorenstein-projective modules, and so have a large supply of perfect paths in view of the classification result of Chen--Shen--Zhou.

Recall that $\Lambda$ is Gorenstein if the two-sided injective dimensions $\idim _\Lambda \Lambda < \infty$ and $\idim \Lambda_\Lambda < \infty$ are finite, in which case they are equal by a theorem of Zaks \cite{Za}. In this case, we denote this common number by $\dim \Lambda$ and
call it the Gorenstein dimension of $\Lambda$. When $\Lambda$ is Gorenstein, \cite[4.2]{Bu} shows that ${\rm Ext}^i_\Lambda(M, \Lambda) = 0$ for all $i >0$ is necessary and sufficient for $M$ to be Gorenstein-projective and so the $n$-th syzygy $\Omega^n N$ of any module $N \in \modsf \Lambda$ is Gorenstein-projective for $n \geq \dim \Lambda$.  

For finite dimensional algebras, by work of Bergh--Jorgensen--Oppermann \cite{BJO} one can in fact characterise the Gorenstein property for $\Lambda$ in terms of the syzygies of $\Bbbk$. We will use this in Section 3 to characterise Gorenstein monomial algebras in terms of Anick chains. The following easy consequence of \cite[Th. 4.1]{BJO} will be the relevant result for us. 
\begin{prop}\label{gorensteindefect} Let $\Lambda = kQ/I$ be a finite-dimensional path algebra. Then $\Lambda$ is Gorenstein if and only if $\Omega^n \Bbbk$ is Gorenstein-projective for some $n \geq 0$. In this case, the minimal such $n$ equals the Gorenstein dimension of $\Lambda$. 
\end{prop} 
\begin{proof}
	If $\Lambda$ is Gorenstein then $\Omega^n \Bbbk$ is Gorenstein-projective whenever $n \geq \dim \Lambda$. Conversely, if $\Omega^n \Bbbk$ is Gorenstein-projective for some $n \geq 0$ then it is easy to see that the \emph{Gorenstein defect category} (that is, the singularity category of $\Lambda$ modulo Gorenstein-projectives, see \cite{BJO}) vanishes by Jordan-H\"older filtration arguments, and this characterises the Gorenstein property by \cite[Th. 4.1]{BJO}. For the last statement, let $\Lambda$ be Gorenstein and let $n$ be the minimal such integer. It is immediate that $n \leq \dim \Lambda$. Since $\Omega^n \Bbbk$ is Gorenstein-projective, we have 
	\begin{align} 
		{\rm Ext}^{n+i}_\Lambda(\Bbbk, \Lambda) = {\rm Ext}^i_\Lambda(\Omega^n \Bbbk, \Lambda) = 0 \textup{ for all } i > 0.  
		\end{align}
From Jordan-H\"older filtrations one obtains that ${\rm Ext}^{n+i}_\Lambda(N, \Lambda) = 0$ for all $i > 0$ and $N \in \modsf \Lambda$, and so $\dim \Lambda \leq n$.
\end{proof}

\subsection{Anick chains over Gorenstein algebras} 

We now know how to test the Gorenstein property for a monomial algebra $\Lambda$ in terms of the syzygy modules $\Omega^n \Bbbk$ via Proposition~\ref{gorensteindefect}. In this section, we translate this into the combinatorial language of Anick chains, and deduce normal forms for chains of sufficiently large degree over Gorenstein monomial algebras.  

Up to now we have taken modules without further qualifiers to mean right modules. However, certain results below will need to be stated for both left and right modules, and similarly results on the structure of Anick chains will need to be stated for both right and left Anick chains, which correspond to the right and left minimal projective resolution of $\Bbbk = \Lambda/\mathfrak{r}$, respectively. In what follows we will continue working with right modules and right Anick chains, and simply note that all dual statements for left modules and chains can be obtained formally by passing to the opposite algebra $\Lambda \mapsto \Lambda^{\rm op}$. 

We start by relating the syzygy modules to Anick chains.

\begin{prop}\label{syzygies} Let $\Lambda$ be a monomial algebra and $n \geq 1$. We have isomorphisms of right and left modules 
	\begin{align}
		&\Omega^n(\Bbbk_\Lambda) \cong \bigoplus_{\upgamma \in C_{n-1}}	t_\upgamma \Lambda \\
		&\Omega^n( _\Lambda \Bbbk) \cong \bigoplus_{\upgamma \in C_{n-1}} \Lambda h_\upgamma. 
	\end{align}
\end{prop}

\begin{proof}
Recall that the minimal resolution of $(P_*, \partial) \xrightarrow{\sim} \Bbbk$ is given by the Anick resolution 
	\begin{align*} 
		\cdots \to	kC_{n-1} \otimes_\Bbbk \Lambda \xrightarrow{\partial_{n-1}} kC_{n-2} \otimes_\Bbbk \Lambda \xrightarrow{\partial_{n-2}} \cdots  \xrightarrow{\partial_1} kC_0 \otimes_\Bbbk \Lambda \xrightarrow{\partial_0} \Lambda \to 0 
	\end{align*} 
with $n$-th term $P_n = kC_{n-1} \otimes_\Bbbk \Lambda$ and differential $\partial_{n-1}(\upgamma) = \upgamma' \otimes t_\upgamma$, and we can identify $\Omega^n \Bbbk = {\rm Im}(\partial_{n-1})$ explicitly. Consider the identification of projective modules 
	\begin{align} 
		kC_{n-1} \otimes_\Bbbk \Lambda = \bigoplus_{\upgamma \in C_{n-1}} e_{s(\upgamma)} \Lambda 
	\end{align}
	given by sending $\upgamma \otimes 1$ to $e_{s(\upgamma)}$. The image of $\partial_{n-1}$ then takes the form 
	\begin{align} 
	{\rm Im}(\partial_{n-1}) = \sum_{\upgamma \in C_{n-1}} t_{\upgamma} \Lambda \subseteq \bigoplus_{\upgamma' \in C_{n-2}} e_{s(\upgamma')} \Lambda 
	\end{align}
	where by abuse of notation we let $\upgamma' \in C_{n-2}$ run over all $(n-2)$-chains, and we write a $(n-1)$-chain as $\upgamma = \upgamma' t_\upgamma$ so that $t_\upgamma \Lambda \subseteq e_{s(\upgamma')} \Lambda$ sits as a submodule in the corresponding copy of $e_{s(\upgamma')}\Lambda$. Moreover, here we must assume that $n \geq 2$, but the case $n = 1$ is treated similarly by replacing the righthand side by $\Lambda$. 

Now note that this sum is in fact direct. Assume that $\upgamma_1 = \upgamma_1' t_{\upgamma_1}$ and $\upgamma_2 = \upgamma_2' t_{\upgamma_2}$ are $(n-1)$-chains and $t_{\upgamma_1}\Lambda$, $t_{\upgamma_2} \Lambda$ are sent into the same summand $e_{s(\upgamma')}\Lambda$, meaning that $\upgamma_1' = \upgamma_2' = \upgamma'$. If $t_{\upgamma_1}\Lambda \cap t_{\upgamma_2} \Lambda \neq 0$, then there is a path left divisible by both tails $t_{\upgamma_1}$ and $t_{\upgamma_2}$, and so one tail must divide the other. As $\upgamma'_1 = \upgamma_2'$, one of $\upgamma_1$, $\upgamma_2$ then divides the other, contradicting minimality. Hence we obtain: 
	\begin{align} 
	{\rm Im}(\partial_{n-1}) = \bigoplus_{\upgamma \in C_{n-1}} t_{\upgamma} \Lambda \subseteq \bigoplus_{\upgamma' \in C_{n-2}} e_{s(\upgamma')} \Lambda, 
	\end{align}
	as required. The second claim is dual. 
\end{proof}

Our next step is to characterise Gorenstein monomial algebras in terms of Anick chains. We first need two auxiliary statements. The first of them collects several results of \cite{CSZ}. 

\begin{lem}[{\cite[Lem.~3.1, Prop.~4.6]{CSZ}}]\label{CSZlemma} Let $\Lambda$ be a monomial algebra and $p$ a non-trivial path in $\Lambda$. Then: 
	\begin{enumerate}[i)]
	\item $p\Lambda$ is Gorenstein-projective and non-projective if and only if $R(p) = \{q\}$ for $q$ a perfect path. 
	\item $\Lambda q$ is Gorenstein-projective and non-projective if and only if $L(q) = \{ p \}$ for $p$ a perfect path. 
	\item $p\Lambda$ is projective if and only if $R(p)$ is empty.
	\item $\Lambda q$ is projective if and only if $L(q)$ is empty. 
\end{enumerate}
\end{lem}
\begin{rem}
	This lemma is ``best possible'' in that Chen-Shen-Zhou construct \cite[Ex.~4.5]{CSZ} an example of a non-perfect path $p$ with $\Lambda p$ Gorenstein-projective, along with an isomorphism $\Lambda p \cong \Lambda p'$ for some perfect path $p'$.  
\end{rem} 

The second result we need is as follows. 

\begin{lem}\label{codepthlemma} Let $\Lambda$ be a Gorenstein algebra, and let $(P_*, \partial) \xrightarrow{\sim} \Bbbk$ be the minimal projective resolution. If $\Omega^n \Bbbk = {\rm Im}(\partial: P_n \to P_{n-1})$ has a non-trivial projective summand, then $n \leq \dim \Lambda$. 
\end{lem}
\begin{proof}Decompose $\Omega^n \Bbbk = N \oplus P$ with $P$ a non-trivial projective summand. Consider the short exact sequence 
	\begin{align}
		0 \to \Omega^n \Bbbk \xrightarrow{\iota} P_{n-1} \xrightarrow{\partial} \Omega^{n-1}\Bbbk \to 0. 
	\end{align}
	If $n > \dim \Lambda$, we show that the embedding $\iota_{| P}: P \hookrightarrow P_{n-1}$ splits, and this will contradict minimality of the resolution. Consider the commutative diagram of short exact sequences 
	\[
		\xymatrix{
			& 0  &&& 0  \\ 
			& N \ar@/_/[d]_{s} \ar[rr]^-{\pi \iota s} \ar[u] && M \ar[ur] \ar@{..>}[d]  \\ 
			0 \ar[r] & \Omega^n \Bbbk \ar@/_/[d] \ar[u] \ar[r]^-{\iota} & P_{n-1} \ar[r] \ar[ur]^-{\pi} & \Omega^{n-1} \Bbbk \ar[r] & 0  \\ 
			& P \ar[u] \ar[ur]^-{\iota_{| P}}  \\ 
			0 \ar[ur] & 0 \ar[u] \\ 
		}
	\]
	A diagram chase shows that we have an induced exact sequence
	\begin{align}
		0 \to N \to M \to \Omega^{n-1} \Bbbk \to 0.	
	\end{align}
	Now, Gorenstein-projectives are closed under summands and extensions, and must be projective the moment they have finite projective dimension. If $n > \dim \Lambda$, then $\Omega^{n-1} \Bbbk$ and $\Omega^n \Bbbk$, and therefore also $N$, are Gorenstein-projective, and thus so is the extension $M$. But the diagonal short exact sequence shows that $\pdim M < \infty$ and so $M$ must be projective. Hence the map $\iota_{| P}: P \hookrightarrow P_{n-1}$ is split as claimed, contradicting minimality of the resolution. 
\end{proof}

Results of Lu--Zhu \cite[Rem. 4.12]{LZ} give a simple description of the self-injective monomial algebras: these are precisely the algebras $\Lambda = \Lambda_{m,n}$ whose quiver is a simple oriented cycle on $m$ vertices with relations consisting of all paths of length $n$. We now show how to characterise the Gorenstein algebras of dimension at most $d+1$ in terms of the structure of $d$-chains; note that this forces $0 \leq d < \infty$, and so our results complement the case of Gorenstein dimension zero of Lu and Zhu. Their paper also characterises monomial algebras of Gorenstein dimension at most~$1$; we leave it to the reader to compare \cite[Th.~5.4]{LZ} with the corresponding particular case of the result below.

\begin{thm}\label{gorensteincharacterisation} Let $\Lambda$ be a monomial algebra. The following are equivalent, for $d\ge 0$: 
	\begin{enumerate}[i)] 
	\item $\Lambda$ is Gorenstein of dimension at most $d+1$. 
	\item For every $d$-chain $\upgamma$, either $R(t_\upgamma)$ is empty or equal to $\{ p \}$ with $p$ perfect. 
	\item For every $d$-chain $\upgamma$, either $L(h_\upgamma)$ is empty or equal to $\{ q \}$ with $q$ perfect. 
	\end{enumerate}
Moreover if $\dim \Lambda < d+1$, then $t_\upgamma$ and $h_\upgamma$ themselves are perfect, and in particular the
sets above are never empty.  
\end{thm} 
\begin{proof}
Prop. \ref{gorensteindefect} and Prop. \ref{syzygies} show that $\Lambda$ is Gorenstein of dimension $\dim \Lambda \leq d+1$ if and only if $\Omega^{d+1} \Bbbk = \bigoplus_{\upgamma \in C_d} t_\upgamma \Lambda$ is Gorenstein-projective, and by Lemma \ref{CSZlemma} this occurs if and only if $R(t_\upgamma)$ is empty or consists of a single perfect path. This shows the equivalence i)-ii), with i)-iii) formally dual. 

For the last claim assume that actually $\dim \Lambda \leq d$. First assume that $d \geq 1$. Let $\upgamma \in C_d$, and write it as $\upgamma = \upgamma' t_{\upgamma}$ for $\upgamma' \in C_{d-1}$. By the above we have $R(t_{\upgamma'}) = \emptyset$ or $R(t_{\upgamma'}) = \{ p \}$ for $p$ perfect. But $t_{\upgamma} \in R(t_{\upgamma'}) \neq \emptyset$, which shows that $t_{\upgamma} = p$ is perfect. Dually $h_\upgamma$ is perfect. 

For the case $d = 0$ we have $\upgamma = t_\upgamma$, which must then be an arrow. Lemma \ref{codepthlemma} shows that $t_\upgamma \Lambda$ cannot be projective and so $R(t_\upgamma) \neq \emptyset$, which forces $R(t_\upgamma) = \{ p \}$ for $p$ perfect. Letting $L(p) = \{ q \}$, we see that $t_\upgamma p = 0$ forces $q$ to be a right divisor of $t_\upgamma$. But $t_\upgamma$ is an arrow and so $t_\upgamma = q$ is perfect, as claimed. The case of $h_\upgamma$ is dual as before. 
\end{proof}

Ignoring the precise dimension, we obtain a cleaner characterisation of Gorenstein monomial algebras. 
\begin{cor} Let $\Lambda$ be a monomial algebra. The following are equivalent: 
	\begin{enumerate}[i)]
	\item $\Lambda$ is Gorenstein.
	\item There exists an $n_0 \in \N$ such that every $n$-chain $\upgamma$ for $n \geq n_0$ has tail $t_\upgamma = p$ given by a perfect path.
	\item There exists an $n_0 \in \N$ such that every $n$-chain $\upgamma$ for $n \geq n_0$ has head $h_\upgamma = q$ given by a perfect path. 
\end{enumerate}
\end{cor}

As an example of how our results can be applied, let us give a characterisation of local Gorenstein monomial algebras. By a local monomial algebra, we mean algebras $\Lambda$ of the form $\Lambda = k\langle S \rangle/I$ with $S$ a non-empty finite set and $I \subseteq (S)^2$ generated by monomial relations. (We still assume $\Lambda$ finite dimensional.) 

\begin{prop}\label{prop:LocalMonomial} Let $\Lambda = k\langle S \rangle/I$ be a local monomial algebra. Then $\Lambda$ is Gorenstein if and only if $\Lambda \cong k[t]/(t^n)$ for some $n \geq 2$. 
\end{prop}

\begin{proof}
We show that $|S| \geq 2$ implies that $\Lambda$ is not Gorenstein. Let $s \neq t \in S$ be distinct elements, and note that $s^m, t^n \in R$ for some minimal $m,n \geq 2$ by finite dimensionality of $\Lambda$. The elements $t, t^n, t^{n+1}, t^{2n}, \dots, t^{kn}, t^{kn+1}, \dots$ are all Anick chains and in particular $t^{kn+1}$ has tail $t$. If $\Lambda$ is Gorenstein then $t^{kn+1}$ has perfect tail for $k \gg 0$ and so $t$ must be a perfect path. Similarly $s$ is perfect. 

We have $(ts)^l \in I$ for $l \gg 0$ and so the set of relations of $\Lambda$ must contain a string of the form $tstst\dots$ or $ststs\dots$, say the first one without loss of generality. The perfect pair $(t,t^{n-1})$ then shows that $t^{n-1}$ left divides $stst\dots$, a contradiction. 
\end{proof}

We were not able to locate a precise reference for Proposition \ref{prop:LocalMonomial} in the literature, although we have no doubts that it is known to experts. In particular, it admits the following short proof not relying on combinatorics of Anick chains. A local Gorenstein algebra $A$ must be self-injective, as follows from the Auslander--Buchsbaum Formula for noncommutative local rings of Wu-Zhang (see \cite[Th. 0.3]{WZ}, using $\Hom_k(A, k)$ for the pre-balanced dualizing complex). Applying \cite[Rem. 4.12]{LZ}, we see that $A$ is a local Gorenstein monomial algebra if and only if $A \cong k[t]/(t^n)$ for some $n$.

\subsection{Example of a Gorenstein monomial algebra}\label{sec:gorensteinexample}

Let $d \geq 2$. Let $\Lambda_d = kQ_d/I_d$, with quiver $Q_d$ as shown in the picture and relations $I_d = (R_d)$ given by $R_d = \{ \beta_1\beta_2,\ \! \beta_2\beta_1,\ \!\alpha_{i} \alpha_{i+1} \ | \ 1\leqslant i \leqslant d-1 \}$. Then $\Lambda_d$ is Gorenstein of dimension $d$. 
\begin{wrapfigure}{r}{0.35\textwidth}
\vspace{-1 em}
\begin{tikzpicture}[auto, scale = 0.55]
    \foreach \a in {0,1,3}
    {
        \node (u\a) at ({\a*45+180}:3){$\bullet$};
        \draw [latex-,
        		line width = 1.15 pt,
        		domain=\a*45+185:\a*45+220] plot ({3*cos(\x)}, {3*sin(\x)});}
        \node (A) at (0,4){$\bullet$};
        \node (B) at (0,2){$\bullet$};
        \draw[-latex,line width = 1.15 pt] (B) to [out=130,in=235](A);
         \draw[-latex,line width = 1.15 pt] (A) to [out=300,in=60](B);
  \node (b1) at (0.95,3){$\beta_2$};
      \node (b2) at (-0.95,3){$\beta_1$};
        \node (d1) at (180:3){$\bullet$};
        \node (u1) at (270:3){$\bullet$};
        \node (u5) at (0:3){$\bullet$};
    \node (a1) at (45:2.15){$\delta_1$};
    \node (a2) at (135:2.15){$\delta_2$};
    \node (a3) at (247.5:2.5){$\alpha_2$};
    \node (a5) at (202.5:2.5){$\alpha_1$};
    \node (a6) at (-22.5:2.5){$\alpha_d$};
    \draw [latex-,	line width = 1.15 pt,
        		domain=45*2+185:45*2+185+15] plot ({3*cos(\x)}, {3*sin(\x)});
    \draw [line width = 1.15 pt,
        		domain=45*2+185+25:45*2+220] plot ({3*cos(\x)}, {3*sin(\x)});
     \draw [dotted, line width = 1.15 pt,
        		domain=45*2+185+15:45*2+185+25] plot ({3*cos(\x)}, {3*sin(\x)});
    \draw [-latex,	line width = 1.15 pt] (B) -- (u5);
    \draw [-latex,	line width = 1.15 pt] (u0) -- (B);
\end{tikzpicture}
\end{wrapfigure}
Indeed all Anick $(d-1)$-chains are of the form $\beta_1 \beta_2 \beta_1 \dots$ or $\beta_2 \beta_1 \beta_2 \dots$ with tails $\beta_{i}$ perfect, or by the path $\alpha_1 \alpha_2 \dots \alpha_d$ whose tail $\alpha_d$ satisfies $R(\alpha_d) = \emptyset$, and so $\Lambda_d$ is Gorenstein of dimension $\leq d$ by Theorem \ref{gorensteincharacterisation}. 
However $\alpha_1 \dots \alpha_{d-1}$ is an $(d-2)$-chain with tail $\alpha_{d-1}$ and $R(\alpha_{d-1}) = \{ \alpha_d \}$ is neither empty nor is $\alpha_d$ perfect, and so $\Lambda_d$ is not Gorenstein of dimension $\leq d-1$. 
Finally, note that $\alpha_d \Lambda_d$ is a non-trivial projective summand in $\Omega^d \Bbbk$, and that such projective summands do not occur in higher degrees as all further Anick chains are given by $\beta_1 \beta_2 \beta_1 \dots$ and $\beta_2 \beta_1 \beta_2 \dots$ with tails $\beta_i$.

\subsection{Perfect walks} In analysing Anick chains of sufficiently high degree over Gorenstein algebras, it will turn out useful to introduce paths which consist of walking along a perfect cycle. 

\begin{defn}A sequence of paths $(p_0, p_1, \dots, p_{l-1})$ is called a perfect walk if it can be extended to a perfect cycle ${\bf p} = (p_0, p_1, \dots, p_{l-1}, p_{l}, \dots, p_{s-1})$. We call $l \geq 1$ the length of the walk. 
\end{defn}

Note that the next and previous paths in a perfect walk are uniquely determined as we have $R(p_i) = \{ p_{i+1} \}$ and $L(p_i) = \{ p_{i-1} \}$, and in particular any perfect path $p = p_0$ can be extended in both directions to a perfect walk of any length
\begin{align}
	\dots, p_{-2}, p_{-1}, p_0, p_1, p_2, \dots	
\end{align}

By abuse of notation we will often say that a path $\alpha = p_0 p_1 \dots p_l$ is given by a perfect walk if $(p_0, p_1, \dots, p_l)$ forms a perfect walk. For instance, the paths $\beta_1 \beta_2 \beta_1\dots$ and $\beta_2 \beta_1 \beta_2\dots$ of example in Section \ref{sec:gorensteinexample} are given by perfect walks. We note that the decomposition of a path as a perfect walk is not in general unique.

\begin{exmp}\label{perfectwalkexample}Let $\Lambda = k[t]/(t^n)$ for $n \geq 2$. Then $t^{kn}$ can be written as a perfect walk in two ways: 
	\begin{align}
		t^{kn} = t \cdot  t^{n-1} \cdots t \cdot  t^{n-1} = t^{n-1} \cdot t \cdots t^{n-1} \cdot t
	\end{align}
	corresponding to the two perfect walks $(t, t^{n-1}, \dots, t, t^{n-1})$ and $(t^{n-1}, t, \dots, t^{n-1}, t)$. 
\end{exmp}

In the above example, the path $t^{kn}$ forms a $(2k-1)$-chain, and one might be tempted to think that all perfect walks will give rise to Anick chains. This is not so:

\begin{exmp} Let $\Lambda = k[t]/(t^n)$ for $n \geq 3$. Consider the perfect cycle $(t^{n-1}, t)$. Then $(t^{n-1}, t, t^{n-1})$ is a perfect walk as it can be extended to $(t^{n-1}, t, t^{n-1}, t)$ but $t^{2n-1} = t^{n-1} \cdot t \cdot t^{n-1}$ is not an Anick chain. 
\end{exmp} 

We will be particularly interested in Anick chains that can be (partially) written as perfect walks, and so we will derive a few properties of such chains. The following lemma easily follows from the definition, we leave the proof to the reader. 
\begin{lem}[Extending Anick chains]\label{extendingchains} Let $\upgamma_n$ be an $n$-chain. Then: 
	\begin{enumerate}[i)]
	\item $\upgamma_{n+1} = \upgamma_n p$ is an $(n+1)$-chain with tail $t_{\upgamma_{n+1}} = p$ if and only if $p \in R(t_{\upgamma_n})$. 
	\item $\upgamma_{n+1} = q \upgamma_n$ is an $(n+1)$-chain with head $h_{\upgamma_{n+1}} = q$ if and only if $q \in L(h_{\upgamma_n})$. 
\end{enumerate}
\end{lem}

\begin{prop}[Weak unique extension property]\label{weakuniqueextensionproperty} Let $\upgamma_n$ be an $n$-chain. Then: 
	\begin{enumerate}[i)] 
	\item If $t_{\upgamma_n} = p_0$ is perfect, then $\upgamma_{n+1} = \upgamma_n p_1$ is the unique right extension of $\upgamma$ as an $(n+1)$-chain, then with tail $p_1$. 
	\item If $h_{\upgamma_n} = q_0$ is perfect, then $q_{-1} \upgamma_n$ is the unique left extension of $\upgamma$ as an $(n+1)$-chain, then with head $q_{-1}$.
\end{enumerate}
\end{prop} 
\begin{proof}
	This follows from Lemma \ref{extendingchains} as $R(p_0) = \{ p_{1} \}$ and $L(q_0) = \{ q_{-1} \}$.  
\end{proof}

\begin{cor}[Unique extension property]\label{uniqueextensionproperty} Let $\upgamma$ be an $n$-chain and $\alpha$ some non-trivial path in $Q$.
	\begin{enumerate}[i)]
	\item If $t_{\upgamma} = p_0$ is perfect, then $\upgamma \alpha$ is an Anick $(n+k)$-chain for some $k \geq 1$ if and only if $\alpha = p_{1} \dots p_{k}$, in which case $t_{\upgamma \alpha} = p_{k}$. 
	\item If $h_{\upgamma} = q_0$ is perfect, then $\alpha \upgamma$ is an Anick $(n+k)$-chain for some $k \geq 1$ if and only if $\alpha = q_{-k} \dots q_{-1}$, in which case $h_{\alpha \upgamma} = q_{-k}$. 
\end{enumerate}
\end{cor} 
\begin{proof} We only prove the first claim as the proof of the second one is similar. The if direction follows from iterating Prop. \ref{weakuniqueextensionproperty}, which shows $\alpha = p_{1} \dots p_{k}$ has the required property. We prove the converse by induction starting with $k = 1$. Clearly if $\upgamma \alpha$ is an $(n+1)$-chain then $\alpha = p_{1}$ by Prop. \ref{weakuniqueextensionproperty}, and then $t_{\upgamma \alpha} = p_{1}$. 

For $k > 1$ we can write $\upgamma \alpha = \eta t_{\upgamma \alpha}$ for an $(n+k-1)$-chain $\eta$. Then one of $\upgamma$ and $\eta$ divide the other, which forces $\eta = \upgamma \alpha'$ as $n < n+k-1$ by minimality of Anick chains. By induction $\alpha' = p_{1} \dots p_{k-1}$ and $t_{\upgamma\alpha'} = p_{k-1}$. Then $\upgamma \alpha = (\upgamma \alpha') t_{\upgamma \alpha}$ is an $(n+k)$-chain extending the $(n+k-1)$-chain $\upgamma \alpha' = \upgamma p_1 \dots p_{k-1}$ and so $t_{\upgamma \alpha} = p_{n+k}$ by the base case. Thus $\alpha = \alpha' t_{\upgamma \alpha} = p_{1} \dots p_{k}$ and $t_{\upgamma \alpha} = p_{k}$ as claimed. 
\end{proof}

From this, we obtain a description of Anick chains of high degree over a Gorenstein monomial algebra. 
\begin{cor}\label{highdegreechains} Let $\Lambda$ be a Gorenstein monomial algebra, let $d = \dim \Lambda$ and let $\upgamma_n$ be an $n$-chain for $n \geq d$. Then $\upgamma_n$ is extended from a $(d-1)$-chain by a perfect walk. More precisely: 
	\begin{enumerate}[i)] 
	\item There is a $(d-1)$-chain $\upgamma_{d-1}$ and a perfect walk $(p_d, p_{d+1}, \dots, p_n)$ such that $\upgamma_{d-1} p_d p_{d+1} \dots p_i$ is an $i$-chain with tail $p_i$ for $i = d, d+1, \dots, n$ and $\upgamma_n = \upgamma_{d-1} p_d p_{d+1} \dots p_n$. 
	\item There is a $(d-1)$-chain $\upgamma_{d-1}$ and a perfect walk $(q_{-n}, q_{-n+1}, \dots, q_{-d})$ such that $q_{-i} q_{-i+1} \dots q_{-d} \upgamma_{d-1}$ is an $i$-chain with head $q_{-i}$ for $i = d, d+1, \dots, n$ and $\upgamma_n = q_{-n} q_{-n+1} \dots q_d \upgamma_{d-1}$.  
\end{enumerate}
\end{cor}
\begin{proof}We prove i) as ii) is dual. Write $\upgamma_n = \upgamma_{n-1} t_{\upgamma_n} = \dots = \upgamma_{d-1} t_{\upgamma_d} t_{\upgamma_{d+1}} \dots t_{\upgamma_n}$ for a $(d-1)$-chain $\upgamma_{d-1}$ and a suitable sequence of tails. Now the $d$-chain $\upgamma_d$ must have a perfect tail by Theorem \ref{gorensteincharacterisation} since $d+1 > d = \dim \Lambda$ and so we may set $t_{\upgamma_d} = p_d$. The unique extension property then gives $t_{\upgamma_{d+1}} = p_{d+1}$, \dots, $t_{\upgamma_n} = p_n$. 
\end{proof}

In other words, this result shows that the structure of $n$-chains over Gorenstein monomial algebras eventually stabilises for $n \geq \dim \Lambda$ and becomes predictable, with the subchain $\upgamma_{d-1}$ of $\upgamma_n$ consisting of noise which we will be able to ignore.

\section{The {\bf Fg} conditions and higher structures}\label{sec:Ainf}

\subsection{Hochschild cohomology and the {\bf Fg} conditions} Let $\Lambda^{\sf e} := \Lambda^{\rm op} \otimes_k \Lambda$ be the enveloping algebra of $\Lambda$, and consider $\Lambda \in \modsf \Lambda^{\sf e}$ with its natural module structure. Recall that the Hochschild cohomology ring ${\rm HH}^*(\Lambda, \Lambda) := {\rm Ext}^*_{\Lambda^{\sf e}}(\Lambda, \Lambda)$ is a graded-commutative algebra in that $x \smile y = (-1)^{|x||y|} y \smile x$ for all elements $x, y \in {\rm HH}^*(\Lambda, \Lambda)$, and we let ${\rm HH}^{\sf ev}(\Lambda, \Lambda)$ denote the subalgebra of elements of even degree. 

For any $M \in \modsf \Lambda$, we have an algebra homomorphism 
\begin{align}
	\varphi_M := M\otimes_\Lambda - : {\rm HH}^*(\Lambda, \Lambda) \to {\rm Ext}^*_\Lambda(M, M).  
\end{align}
The map $\varphi_M$ induces on ${\rm Ext}^*_\Lambda(M, N)$ for any $N \in \modsf \Lambda$ the structure of a right graded module over Hochschild cohomology; moreover this is compatible with the left action coming from $\varphi_N$ as these satisfy $\varphi_N(x) \cdot \theta = (-1)^{|x| |\theta|} \theta \cdot \varphi_M(x)$ for all elements $\theta \in {\rm Ext}^*_\Lambda(M, N)$ and $x \in {\rm HH}^*(\Lambda, \Lambda)$. This further restricts to a right module structure over any subalgebra ${\rm H} \subseteq {\rm HH}^*(\Lambda, \Lambda)$. The finite generation ({\bf Fg}) conditions of Snashall--Solberg are then stated as follows: 
\begin{enumerate}[${\bf Fg}\ \! 1$.] 
	\item ${\rm HH}^{\sf ev}(\Lambda, \Lambda)$ contains a Noetherian graded subalgebra ${\rm H}$ with ${\rm H}^0 = {\rm HH}^0(\Lambda, \Lambda)$. 
	\item ${\rm Ext}^*_\Lambda(M, N)$ is a finitely generated ${\rm H}$-module for all $M, N \in \modsf \Lambda$. 
	\end{enumerate} 

	We say that $\Lambda$ satisfies ${\bf Fg}$ if both conditions above hold. The ${\bf Fg}$ conditions imply that ${\rm HH}^*(\Lambda, \Lambda)$ is module-finite over ${\rm H}$, and so that ${\rm HH}^*(\Lambda, \Lambda)$ is itself a Noetherian algebra. Moreover, in the presence of ${\bf Fg}\ \! 1$, by Jordan-H\"older filtration arguments it is enough to establish ${\bf Fg}\ \! 2$ in the case $M = N = \Bbbk$. 

\subsection{$A_\infty$-centres of minimal $A_\infty$-algebras}\label{Ainfcentre} 
In \cite{BG} the image of the characteristic homomorphism 
\begin{align}
\varphi_M: {\rm HH}^*(\Lambda, \Lambda) \to {\rm Ext}^*_\Lambda(M, M)
\end{align} 
was related for any $M \in \modsf \Lambda$ to the $A_\infty$-structure on ${\rm Ext}^*_\Lambda(M, M)$, where it was shown to consist of $A_\infty$-central classes. Recall that a class $a \in E = (E, m_2)$ in a graded algebra is graded-central if $m_2(a,x) = (-1)^{|a||x|}m_2(x,a)$ for all $x \in E$. Equivalently, the inner derivation ${\rm ad}_a$ vanishes:
\begin{align}
	{\rm ad}_a(x) = [a,x] = m_2(a,x) - (-1)^{|a||x|}m_2(x,a) = 0.
\end{align}
We denote by $\mathcal{Z}_{gr} E = \{ a \in E \ | \ {\rm ad}_a = 0 \}$ the graded centre. 

Generalising to $E = (E, \{ m_n \}_{n \geq 2})$ a minimal $A_\infty$-algebra, given $a \in E$, for each $n \geq 1$ we form the $n$-ary higher commutator 
\begin{align}
	[a;x_1, \dots, x_n]_{1,n} := \sum_{i=0}^{n} (-1)^{i} (-1)^{|a|(|x_1| + \dots + |x_{i}|)}\ \! m_{n+1}(x_1, \dots, x_{i}, a, x_{i+1}, \dots, x_n).
\end{align}
In other words we apply $m_{n+1}$ to the (signed) shuffle product $a \sh (x_1 \otimes \dots \otimes x_n)$. We then define the homotopy inner derivation ${\rm ad}_a = \{ {\rm ad}_{a, n} \}_{n \geq 1}$ to be a collection of $n$-ary operations ${\rm ad}_{a,n}: E^{\otimes n} \to E $ given by ${\rm ad}_{a,n}(x_1, \dots, x_n) := [a; x_1, \dots, x_n]_{1,n}$. The collection ${\rm ad}_a$ forms a cocycle in the complex of \emph{homotopy derivations} 
\begin{align}
	{\rm ad}_a \in {\rm hoder}(E) := \left({\rm Hom}_\Bbbk(\oplus_{n \geq 1} \overline{E}^{\otimes n}, E), \partial \right) 
\end{align} 
which is a subcomplex of the Hochschild cochain complex $C^*(E, E)$ of the $A_\infty$-algebra $E$, see \cite{BG} for details.  

\begin{defn} The $A_\infty$-centre of $E$ is the space \[\mathcal{Z}_\infty E := \{ a \in E \ | \ [{\rm ad}_a] = 0 \text{ in } {\rm H}^*({\rm hoder}(E))\}.\] 
\end{defn}
The $A_\infty$-centre $\mathcal{Z}_\infty E \subseteq E$ is a graded subalgebra of the underlying graded algebra $E = (E, m_2)$, and $A_\infty$-central classes are always graded-central, so that we have containment $\mathcal{Z}_\infty E \subseteq \mathcal{Z}_{gr} E$ with equality whenever $m_n = 0$ for all $n \geq 3$. 

The image of the characteristic homomorphism $\varphi_\Bbbk: {\rm HH}^*(\Lambda, \Lambda) \to {\rm Ext}^*_\Lambda(\Bbbk, \Bbbk)$ is always contained in the graded centre 
\begin{align}
	{\rm im}(\varphi_\Bbbk) \subseteq \mathcal{Z}_{gr}{\rm Ext}^*_\Lambda(\Bbbk, \Bbbk)	 
\end{align}
with equality in the case of $\Lambda$ Koszul by a theorem of Buchweitz--Green--Snashall--Solberg, but with proper inclusion in general. Since the Koszul case corresponds to the situation where one can take $m_n = 0$ for all $n \geq 3$ for the $A_\infty$-structure on ${\rm Ext}^*_\Lambda(\Bbbk, \Bbbk)$, this result was refined in \cite{BG} as follows: 
\begin{thm}[\cite{BG}] The image of the characteristic homomorphism is precisely the $A_\infty$-centre: 
	\begin{align}
		{\rm im}(\varphi_\Bbbk) = \mathcal{Z}_\infty {\rm Ext}^*_\Lambda(\Bbbk, \Bbbk). 
	\end{align}
\end{thm} 
Reading off the $A_\infty$-central condition $[{\rm ad}_a] = 0$ in the cohomology ${\rm H}^*({\rm hoder}(E, E))$ of $E = {\rm Ext}^*_\Lambda(\Bbbk, \Bbbk)$ can be subtle, and so in practice one may focus on the simpler sufficient condition that ${\rm ad}_a$ vanishes on the nose. Equivalently, for $a \in {\rm Ext}^*_\Lambda(\Bbbk, \Bbbk)$, we are interested in the vanishing of higher commutators 
\begin{align}\label{vanishingcommutators} 
	\sum_{i=0}^{n}(-1)^i (-1)^{|a|(|x_1| + \dots + |x_i|)}\ \! m_{n+1}(x_1, \dots, x_i, a, x_{i+1}, \dots, x_n) = 0 
\end{align}
for all $x_1, \dots, x_n \in {\rm Ext}^*_\Lambda(\Bbbk, \Bbbk)$ and all $n \geq 1$. We record this as a corollary:
\begin{cor}\label{sufficientcriterion} Let $a \in {\rm Ext}^*_\Lambda(\Bbbk, \Bbbk)$ be such that all higher commutators (\ref{vanishingcommutators}) vanish. Then $a$ is in the image of $\varphi_\Bbbk: {\rm HH}^*(\Lambda, \Lambda) \to {\rm Ext}^*_\Lambda(\Bbbk, \Bbbk)$. 
\end{cor}

Finally, we can recast the ${\bf Fg}$ conditions in terms of $\mathcal{Z}_\infty {\rm Ext}^*_\Lambda(\Bbbk, \Bbbk)$. For that, we introduce another set of conditions ${\bf Fg'}$: 
\begin{enumerate}[${\bf Fg'}\ \!1$.] 
	\item $\mathcal{Z}_\infty := \mathcal{Z}_\infty {\rm Ext}^*_\Lambda(\Bbbk, \Bbbk)$ is a Noetherian algebra.
	\item ${\rm Ext}^*_\Lambda(\Bbbk, \Bbbk)$ is module-finite over $\mathcal{Z}_\infty$. 
\end{enumerate}

\begin{prop}\label{Fgequivalence} The conditions ${\bf Fg}$ and ${\bf Fg'}$ are equivalent.
\end{prop} 
\begin{proof}
	If $\Lambda$ satisfies ${\bf Fg}$ with regards to ${\rm H} \subseteq {\rm HH}^*(\Lambda, \Lambda)$ then $\varphi_\Bbbk({\rm H}) \subseteq \mathcal{Z}_\infty$ is a Noetherian algebra over which ${\rm Ext}^*_\Lambda(\Bbbk, \Bbbk)$ is module-finite and so ${\bf Fg'}$ holds. 

	Conversely if $\Lambda$ satisfies ${\bf Fg'}$, then $\mathcal{Z}_\infty$ is Noetherian and must be finitely generated as it is positively graded. Moreover, so must be its even subalgebra $\mathcal{Z}_\infty^{\sf ev}$. Picking lifts of said generators, we can form a Noetherian subalgebra ${\rm H} \subseteq {\rm HH}^{\sf ev}(\Lambda, \Lambda)$ satisfying ${\bf Fg}\ \! 1$, possibly after adding ${\rm H}^0 = {\rm HH}^0(\Lambda, \Lambda)$. Since $\varphi_\Bbbk({\rm H}) = {\rm Z}_\infty^{\sf ev}$, one sees from ${\bf Fg'}\ \! 2$ that ${\rm Ext}^*_\Lambda(\Bbbk, \Bbbk)$ is module-finite over ${\rm H}$ and the same then holds for ${\rm Ext}^*_\Lambda(M, N)$ for all $M, N \in \modsf \Lambda$ by Jordan--H\"older filtration arguments. Hence ${\bf Fg}$ holds. 
\end{proof}

Let us remark that this approach to studying the ${\bf Fg}$ condition is not new, see \cite[Th.~1.3]{ES11} for the Koszul case where $\mathcal{Z}_{gr}{\rm Ext}^*_\Lambda(\Bbbk, \Bbbk) = \mathcal{Z}_\infty {\rm Ext}^*_\Lambda(\Bbbk, \Bbbk)$.

\subsection{Combinatorics of the $A_\infty$-centre} In the case of monomial algebras, it is possible to use Proposition \ref{vanishingpatterns} to simplify the general formula for higher commutators.
\begin{lem}\label{simplifiedcommutator} Suppose $a \in {\rm Ext}^*_\Lambda(\Bbbk, \Bbbk)$ is of even degree. Then for any $x_1, \dots, x_n$ in  ${\rm Ext}^*_\Lambda(\Bbbk, \Bbbk)$, the higher commutators simplify to 
\begin{align}
[a;x_1, \dots, x_n]_{1,n} = m_{n+1}(a, x_1, \dots, x_n) + (-1)^{n} m_{n+1}(x_1, \dots, x_n, a). 
\end{align}
\end{lem}
\begin{proof}
Indeed, if the degree of $a$ is even, we have $m_{n+1}(x_1, \dots, x_i, a, x_{i+1}, \dots, x_n) = 0$ for $0 < i < n$. 
\end{proof}

Using this lemma, we obtain a combinatorial criterion for $A_\infty$-centrality. Given elements $a_1, \dots, a_m$ with $a_i \in {\rm Ext}^{r_i+1}_\Lambda(\Bbbk, \Bbbk) = kC_{r_i}^{\vee}$ with $r_i \geq 0$, we interpret the expression $a_1 \dots a_m \in kC_{r_1 + \dots + r_m + 1}^{\vee}$ as denoting the sum of terms of $a_1 \dots a_m \in kQ^{op}$ which are in $kC_{r_1 + \dots + r_m + 1}^\vee$, that is as the projection on the natural summand. Next, we say that an element $a \in {\rm Ext}^*_\Lambda(\Bbbk, \Bbbk)$ is symmetric if $a \cdot e = e \cdot a$ for all $e \in \Bbbk$, i.e. $a$ is a linear combination of closed oriented cycles in the quiver for Ext. The combinatorial criterion then states: 
	\begin{prop}[Combinatorial criterion]\label{combinatorialcriterion} Suppose $a \in {\rm Ext}^*_\Lambda(\Bbbk, \Bbbk)$ is a symmetric element of even degree $|a| = r_0+1 \geq 2$. Assume that for each tuple of chains $c_1, \dots, c_n$ with $c_i \in C_{r_i}^{\vee}$, we have equalities 
		\begin{align} 
		ac_1 \dots c_n = c_1 \dots c_n a
		\end{align}
		in $kC_{r_0 + r_1 + \dots + r_n + 1}^{\vee}$. Then $a$ is $\infty$-central and so lies in the image of Hochschild cohomology. 
\end{prop} 
\begin{proof} 
By Corollary \ref{sufficientcriterion} it's enough to show that $[a; x_1, \dots, x_n]_{1,n} = 0$ for all $x_1, \dots, x_n$ in ${\rm Ext}^*_\Lambda(\Bbbk, \Bbbk)$, and by linearity it's enough to do this for chains $c_1, \dots, c_n$. Since $a c_1 \dots c_n = c_1 \dots c_n a$, we deduce that 
\begin{align*}
	m_{n+1}(a, c_1, \dots, c_n) = \pm m_{n+1}(c_1, \dots, c_n, a) 
\end{align*}
and so the two products are zero or non-zero simultaneously. If both are zero we are done by Lemma \ref{simplifiedcommutator}, and if both are nonzero the elements $c_1, \dots, c_n \in {\rm Ext}^*_\Lambda(\Bbbk, \Bbbk)$ have odd degree by Proposition \ref{vanishingpatterns}. Since $kC_{r_i}^{\vee} = {\rm Ext}^{r_i+1}_\Lambda(\Bbbk, \Bbbk)$, we deduce that $r_1, \dots, r_n$ are even. We can then compute the sign precisely: 
		\begin{align*}
			m_{n+1}(a, c_1, \dots, c_n) &= (-1)^{(n+1)r_0 + r_0r_n + r_0 + r_n}\ \! ac_1 \dots c_n \\ 
			&= (-1)^{(n+1)r_0 + r_0}\ \! ac_1 \dots c_n \\
			&= (-1)^{(n+1) + r_0}\ \! c_1 \dots c_n a\\ 
			&= (-1)^{n+1}(-1)^{r_0}\ \! c_1 \dots c_n a\\ 
			&= (-1)^{n+1} (-1)^{(n+1)r_1 + r_1r_0 + r_1 + r_0}\ \! c_1 \dots c_n a\\ 
			&= (-1)^{n+1} m_{n+1}(c_1, \dots, c_n, a).
		\end{align*} 
		Lemma \ref{simplifiedcommutator} then gives 
		\begin{align*}
			[a;c_1, \dots, c_n]_{1,n} = m_{n+1}(a, c_1, \dots, c_n) + (-1)^{n} m_{n+1}(c_1, \dots, c_n, a) = 0, 
		\end{align*}
which is what we wanted.	\end{proof}

\section{Periodicity operators for Gorenstein algebras}\label{sec:Period}


We now turn to the construction of $A_\infty$-central operators $\chi_\upgamma \in {\rm Ext}^*_\Lambda(\Bbbk, \Bbbk)$ associated to special Anick chains $\upgamma$ which we will call \emph{stable relation cycles}. In this subsection we will always assume that $\Lambda$ is Gorenstein. 

We define the \emph{period} of a Gorenstein monomial algebra $\Lambda$ as follows. The algebra $\Lambda$ contains finitely many perfect paths and so there are finitely many indecomposable Gorenstein-projectives $M$ over $\Lambda$, each with a periodic minimal projective resolution. We let $p_M$ denote the minimal period of this resolution. 

\begin{defn}The period of $\Lambda$ is the least common multiple of all periods $p_M$, as $M$ runs over indecomposable non-projective Gorenstein-projectives. We denote that number by $\ell$.
\end{defn}

One notes that this is unchanged if we replace right modules by left modules as the duality $M \mapsto M^*$ preserves the period of Gorenstein-projectives. 


\begin{defn}[Stable relation cycles] Let $\upgamma$ be an $(s-1)$-chain. We say that $\upgamma$ is a stable relation cycle if:
	\begin{enumerate}[i)]
	\item $s \geq \dim \Lambda + 1$ and $s$ is even. 
	\item $s$ is a multiple of $\ell$. 
	\item $\upgamma = p_0 p_1 \dots p_{s-1}$ for a perfect cycle ${\bf p} = (p_0, p_1, \dots, p_{s-1})$ with tail $t_\upgamma = p_{s-1}$.
	\item $\upgamma = q_0 q_1 \dots q_{s-1}$ for a perfect cycle ${\bf q} = (q_0, q_1, \dots, q_{s-1})$ with head $h_{\upgamma} = q_0$.
\end{enumerate}
\end{defn}

Let us unpack the definition.  

Properties iii) and iv) show that stable relation cycles enjoy the unique extension property of Corollary \ref{uniqueextensionproperty}. Specifically, for any $k \geq 1$, the path $\upgamma p_{s} \dots p_{s-1 + k}$ is the unique extension of $\upgamma$ to the right as an $(s-1+k)$-chain, then with tail $t_{\upgamma p_s \dots p_{s-1+k}} = p_{s-1+k}$, and dually $q_{-k} \dots q_{-1} \upgamma$ is the unique extension to the left as an $(s-1+k)$-chain, then with head $h_{q_{-k} \dots q_{-1} \upgamma} = q_{-k}$. Note that the decomposition iii) and iv) can differ as seen in Example \ref{perfectwalkexample}. 

The unique extension property of stable relation cycles $\upgamma$ will give us control over multiplication against $\upgamma^\vee$ in ${\rm Ext}^*_\Lambda(\Bbbk, \Bbbk)$. However $\upgamma^\vee$ is not a central class in general, and the additional constraints i-ii) will allow us to ``rotate'' $\upgamma$ without breaking the constraints iii)-iv), to obtain new stable relation cycles whose total sum will form an $A_\infty$-central class $\chi_{\upgamma}$ in ${\rm Ext}^*_\Lambda(\Bbbk, \Bbbk)$. 

Green, Snashall and Solberg have introduced in \cite{GSS06} a closely related notion of stability for relation cycles which is weaker than
what we consider here; theirs is concerned with the stability of tails and heads under taking power, which also follows from our definition.
However the precise decomposition in terms of perfect paths is what gives us the additional leverage needed to study higher commutators
and the finite generation conditions. The whole apparatus requires a careful balancing act, and so we begin with studying properties of perfect walks. 

\subsection{Perfect walks of even length}

Now, in general if $(w_0, w_1, \dots, w_{n-1})$ is a perfect walk of even length $n$, then the path $w_0 w_1 \dots w_{n-2} w_{n-1}$ is a concatenation of relations $r_0 = w_0 w_1$, $r_2 = w_2 w_3$, \dots, $r_{n-2} = w_{n-2} w_{n-1}$, and so we have $w_0 w_1 \dots w_{n-1} = r_0 r_2 \dots r_{n-2}$. In particular stable relation cycles are concatenations of relations. Our next aim will be to recognise stable relation cycles amongst perfect walks of even length. This will come through a series of lemmas. We remark that Lemma \ref{evenwalks} below is closely related to \cite[Prop. 2.3]{GSS06}.

\begin{lem}\label{evenwalks} Let $(w_0, w_1, \dots, w_{n-1})$ be a perfect walk of even length $n$. Then $\upgamma = w_0 w_1 \dots w_{n-1}$ is an $(n-1)$-chain and $w_{n-1}$ is a right divisor of $t_{\upgamma}$.  
\end{lem}
\begin{proof}
	Write $n = 2k$, we work by induction on $k \geq 1$. For $k = 1$, $\upgamma = w_0 w_1 \in R = C_1$ is a relation and $w_0 w_1 = \upgamma = a t_{\upgamma}$ for some arrow $a$, so that $w_1$ right divides $t_{\upgamma}$ for length reasons. Next, consider $\upgamma' = w_0 w_1 \dots w_{2k-2}$ which is a chain by induction, and for which $w_{2k-2}$ right divides $t_{\upgamma'}$. This gives $t_{\upgamma'} w_{2k-1} = 0$ in $\Lambda$, so there is a left divisor $w'$ of $w_{2k-1} = w' w''$ such that $\upgamma' w'$ is a chain with tail $w'$.  

	Consider $w'' w_{2k}$, and note that $w' w'' w_{2k} = w_{2k-1} w_{2k} = 0$ in $\Lambda$. Hence we can find a left divisor $u$ of $w'' w_{2k}$ such that $w' u = 0$ in $\Lambda$ and such that $\upgamma' w' u$ forms a chain. Since $w' w'' = w_{2k-1} \neq 0$ in $\Lambda$, we have $u = w'' u'$ and in particular $u'$ left divides $w_{2k}$. But we have $w_{2k-1} u' = w' w'' u' = w' u =  0$, which shows that $w_{2k}$ left divides $u'$, thus forcing $w_{2k} = u'$. 

	Putting it together, this gives a chain $\upgamma' w' u = \upgamma' w' w'' u' = \upgamma' w_{2k-1} w_{2k}$ with tail $u$ right divisible by $w_{2k}$, as claimed.  
\end{proof}

\begin{lem}\label{uniquenessofrelations} Let $\upgamma = w_0 w_1 \dots w_{n-1}$ be an $(n-1)$-chain given by a perfect walk of even length $n$. Then:

	\begin{enumerate}[i)]
	\item The relations $r_0, r_2, \dots, r_{n-2}$ are independent of the choice of perfect walk $(w_0, w_1, \dots, w_{n-1})$ with $\upgamma = w_0 w_1 \dots w_{n-1}$. 
	\item If $\upgamma' = w'_0 w'_1 \dots w'_{n-1}$ is another perfect walk of even length $n$ with $r'_0 = r_0$ or $r'_{n-2} = r_{n-2}$, then $\upgamma' = \upgamma$. 
\end{enumerate}
\end{lem}
\begin{proof}i). If $\upgamma = r_0 r_2 \dots r_{s-2} = \widetilde{r}_0 \widetilde{r}_2 \dots \widetilde{r}_{s-2}$ is written in two ways as a concatenation of relations, then one of $r_0, \widetilde{r}_0$ divides the other and so $r_0 = \widetilde{r}_0$. Applying the same argument inductively to the substrings $r_{2k} \dots r_{s-2} = \widetilde{r}_{2k} \dots \widetilde{r}_{s-2}$ gives the result. 
 
ii). We assume that $r'_{n-2} = r_{n-2}$, the case $r'_{0} = r_{0}$ is similar. We will show that the previous relation in a perfect walk of even length is uniquely determined. 

The hypothesis gives $w'_{n-2} w'_{n-1} = r'_{n-2} = r_{n-2} = w_{n-2} w_{n-1}$, so one of $w'_{n-2}, w_{n-2}$ left divides the other, say $w_{n-2} = w'_{n-2} x$. Then $w'_{n-3} w_{n-2} = w'_{n-3} w'_{n-2} x = 0$ in $\Lambda$. As $w'_{n-3} \neq 0$ in $\Lambda$, we must have $w'_{n-3} = y w_{n-3}$ since $(w_{n-3}, w_{n-2})$ forms a perfect pair. Continuing, we have $w'_{n-4} w'_{n-3} = w'_{n-4} y w_{n-3} = 0$ in $\Lambda$ and so $w'_{n-4} y = z w_{n-4}$. We obtain the equality $w'_{n-4} w'_{n-3} = w'_{n-4} y w_{n-3} = z w_{n-4} w_{n-3}$, so that $z$ is the empty word and we have equality of relations $r'_{n-4} = w'_{n-4} w'_{n-3} = w_{n-4} w_{n-3} = r_{n-4}$. Iterating, we get $r'_{2k} = r_{2k}$ for all $k$ and so $\upgamma' = \upgamma$. 
\end{proof}

The next result allows us to rewrite perfect walks of even length in normal forms. 
\begin{lem}[Perfect rewriting]\label{perfectrewritinglemma} Let $(w_0, w_1, \dots, w_{n-1})$ be a perfect walk of even length $n$ and let $\upgamma = w_0 w_1 \dots w_{n-1}$. 
	\begin{enumerate}[i)]
	\item If $t_\upgamma $ is perfect, then $\upgamma = p_0 \dots p_{n-1}$ is given by the perfect walk with $p_{n-1} = t_\upgamma$. 
	\item If $h_\upgamma $ is perfect, then $\upgamma = q_0  \dots q_{n-1}$ is given by the perfect walk with $q_0 = h_\upgamma$.
\end{enumerate}
\end{lem}
\begin{proof} We prove i) as ii) is dual. Assuming $t_{\upgamma} = p_{n-1}$ perfect, since both $p_{n-1}$ and $w_{n-1}$ right divide $\upgamma$, one must divide the other, say $p_{n-1} = x w_{n-1}$ for some $x$ (the other case is similar). Then $p_{n-2} p_{n-1} = p_{n-2} x w_{n-1} = 0$ in $\Lambda$. Since $p_{n-2} x \neq 0$ in $\Lambda$ as it is a proper divisor of $p_{n-2} p_{n-1}$, we must have $p_{n-2} x = y w_{n-2}$ for some $y$, which gives $p_{n-2} p_{n-1} = p_{n-2} x w_{n-1} = y w_{n-2} w_{n-1}$. It follows that $y$ is the empty word and $p_{n-2} p_{n-1} = w_{n-2} w_{n-1}$.

Now $\upgamma' = p_0 p_1 \dots p_{n-2} p_{n-1}$ and $\upgamma = w_0 w_1 \dots w_{n-2} w_{n-1}$ share the last relation $p_{n-2} p_{n-1} = w_{n-2} w_{n-1}$, and so $\upgamma' = \upgamma$ by Lemma \ref{uniquenessofrelations} ii). 
\end{proof}

We now obtain a recognition lemma for stable relation cycles. 

\begin{lem}[Recognition Lemma]\label{recognitionlemma} Let $(w_0, w_1, \dots, w_{n-1})$ be a perfect walk, and assume that $n \geq \dim \Lambda + 1$, $n$ is even and divisible by $\ell$. Then $\upgamma = w_0 w_1 \dots w_{n-1}$ is a stable relation cycle. 
\end{lem} 
\begin{proof}
	First $\upgamma$ is an $(n-1)$-chain by Lemma \ref{evenwalks}. Next, since $n-1 \geq \dim \Lambda$ we see that $p_{n-1} := t_\upgamma$ and $q_0 := h_\upgamma$ are perfect by Theorem \ref{gorensteincharacterisation}, and Lemma \ref{perfectrewritinglemma} shows that $\upgamma = p_0 p_1 \dots p_{n-1} = q_0 q_1 \dots q_{n-1}$. Lastly, since $n$ is a multiple of the period $\ell$, it is then a multiple of the minimal period of every perfect cycle and so we see that ${\bf p} = (p_0, p_1, \dots, p_{n-1})$ and ${\bf q} = (q_0, q_1, \dots, q_{n-1})$ are in fact perfect cycles. 
\end{proof}

\begin{cor}\label{perfpathstabilise} Let $p$ be a perfect path in $\Lambda$. Then $p = p_0$ $($resp. $p = p_{s-1})$ can be extended to a stable relation cycle $\upgamma = p_0 p_1 \dots p_{s-1}$. 
\end{cor}
\begin{proof}
	Simply take the perfect walk $(p_0, p_1, \dots, p_{s-1})$ of even length $s$, satisfying $s \geq \dim \Lambda + 1$ and with $s$ some multiple of $\ell$. 
\end{proof}

\subsection{Rotating stable relation cycles}

Now let $\upgamma = a_0 a_1 \dots a_{n-1} \in C_{s-1}$ be a stable relation cycle, which we write as a path of length $n$ with $a_i \in Q_1$. We now want to ``rotate'' $\upgamma$ to produce new stable relation cycles, also of path length $n$. We can think of $\upgamma$ as a substring of a periodic string, infinite in both directions: 
\begin{align}
	\underline{a}:	\dots a_0 a_1 \dots a_{n-1} (a_0 a_1 \dots a_{n-1}) a_0 a_1 \dots a_{n-1} \dots 	
\end{align}
One can recover this string by periodicity from any length $n$ substring. Setting $\upgamma_0 = \upgamma$, we define the set $S = \{ \upgamma_0 , \upgamma_1, \dots, \upgamma_{n_{\upgamma} - 1} \}$ by the property that: 
\begin{enumerate}[i)] 
\item $\upgamma_i = a_{j_i}a_{j_i + 1} \dots a_{j_i + n-1}$ is an $(s-1)$-chain which is also a stable relation cycle. 
\item $0 = j_0 < j_1 < \dots < j_{n_\upgamma - 1} \leq n-1$. 
\item For $\upgamma_i = a_{j_i} a_{j_i+1} \dots a_{j_i+n-1} \in S$, we have $\upgamma_i \neq a_{j} a_{j+1} \dots a_{j+n-1}$ if $0 \leq j < j_i$. 
\item $\{j_0, j_1, \dots, j_{n_\upgamma - 1} \} \subseteq \{ 0, 1, \dots, n-1 \}$ is maximal with these property.  
	\end{enumerate}
	We think of $n_\upgamma$ as the number of possible distinct rotations of $\upgamma$ as a stable relation cycle. Moreover, all $\upgamma_i$ are length $n$ substrings of $\underline{a}$, and it is clear that the same set is produced by this construction starting from any other $\upgamma_i \in \{ \upgamma_0, \upgamma_1, \dots, \upgamma_{n_{\upgamma}-1} \}$. We call this set $S$ the associated set of $\upgamma$, which is also the equivalence class for an equivalence relation on stable relation cycles which we denote $\upgamma \sim \upgamma'$. 

	\begin{lem}[Completeness Lemma]\label{completeness} Let $\upgamma \in C_{s-1}$ be a stable relation cycle, and write $\upgamma = p_0 p_1 \dots p_{s-1} = q_0 q_1 \dots q_{s-1}$ for the corresponding perfect walks. Then for any $k \in \Z$, both $p_k p_{k+1} \dots p_{k+s-1}$ and $q_k q_{k+1} \dots q_{k+s-1}$ are in $\{ \upgamma_0, \upgamma_1, \dots, \upgamma_{n_\upgamma - 1} \}$. 
	\end{lem}
	\begin{proof}
		Both strings are substrings of $\underline{a}$, and the periodicities $p_{i+s} = p_i$ and $q_{i+s} = q_i$ shows that if $\upgamma$ is a substring of length $n$ in $\underline{a}$, then so are $p_k p_{k+1} \dots p_{k+s-1}$ and $q_k q_{k+1} \dots q_{k+s-1}$ as these differ by cyclic shifts, which preserves path length. Finally both strings are given by perfect walks of even lengths $s$ satisfying the hypotheses of Lemma \ref{recognitionlemma}, and so $p_k p_{k+1} \dots p_{k+s-1}$ and $q_k q_{k+1} \dots q_{k+s-1}$ are both relation cycles. 
	\end{proof}


	\begin{lem}[Rigidity Lemma]\label{rigidity} Let $\alpha$ be a path in $Q$ and let $k \geq 1$. Then $\upgamma_i \alpha$ (resp. $\alpha \upgamma_i$) is an $(s-1+k)$-chain for at most one choice of $i \in \{ 0, 1, \dots, n_{\upgamma}-1 \}$. 
	\end{lem} 
	\begin{proof} Let $\upgamma_i, \upgamma_{i'}$ be two stable relation cycles in the associated set, and write $\upgamma_i = p_0 p_1 \dots p_{s-1}$ with tail $t_{\upgamma_i} = p_{s-1}$ for the corresponding perfect cycle (resp. $\upgamma_{i'} = p'_0 p'_1 \dots p'_{s-1}$ with tail $t_{\upgamma_{i'}} = p'_{s-1}$). If $\upgamma_i \alpha$ and $\upgamma_{i'} \alpha$ are both $(s-1+k)$-chains, then the unique extension property shows that $\alpha = p_{s} p_{s+1} \dots p_{s-1+k} = p'_{s} p'_{s+1} \dots p'_{s-1+k}$, and by $s$-periodicity we can rewrite this as $\alpha = p_0 p_1 \dots p_{k-1} = p'_0 p'_1 \dots p'_{k-1}$. 

		If $k = 1$ we are done as $p_0 = \alpha = p'_0$, and the perfect walks $(p_0, p_1, \dots, p_{s-1})$ and $(p'_0, p'_1, \dots, p'_{s-1})$ are equal the moment $p_k = p'_k$ for any $k$, so that $\upgamma_i = \upgamma_{i'}$. 

	If $k \geq 2$, then the equality $\alpha = p_0 p_1 \dots p_{k-1} = p'_0 p'_1 \dots p'_{k-1}$ shows that one of the two relations $r_0 = p_0p_1$, $r'_0 = p'_0 p'_1$ divides the other, so that $r_0 = r'_0$. Since stable relation cycles are determined by their first relation (Lemma \ref{uniquenessofrelations}) then $\upgamma_i = \upgamma_{i'}$. \end{proof}

\subsection{Periodicity operators} We can now construct the operators $\chi_\upgamma \in {\rm Ext}^*_\Lambda(\Bbbk, \Bbbk)$ associated to stable relation cycles. 

	\begin{defn}[Periodicity operators] Let $\upgamma \in C_{s-1}$ be a stable relation cycle and $\{ \upgamma_0, \upgamma_1, \dots, \upgamma_{n_\upgamma - 1} \}$ the associated set. Then 
		\begin{align} 
			\chi_\upgamma = \sum_{i=0}^{n_\upgamma - 1} \upgamma_i^\vee 
		\end{align}
		is the periodicity operator $\chi_\upgamma \in {\rm Ext}^s_\Lambda(\Bbbk, \Bbbk)$ attached to $\upgamma$. 
	\end{defn} 

	
\begin{exmp} Let $\Lambda = k[t]/(t^n)$ for $n \geq 2$. Then $\upgamma = t^n \in C_1$ is the unique stable relation cycle of minimal length. The periodicity operator $\chi_\upgamma = (t^n)^\vee \in {\rm Ext}^2_\Lambda(\Bbbk, \Bbbk)$ is the classical periodicity operator of Gulliksen \cite{Gu}. 
	\end{exmp}

	Finally, we can show that the operators $\chi_\upgamma \in {\rm Ext}^*_\Lambda(\Bbbk, \Bbbk)$ lift to Hochschild cohomology. 
	\begin{prop} Let $\Lambda$ be a Gorenstein monomial algebra and let $\upgamma \in C_{s-1}$ be a stable relation cycle. Then $\chi_\upgamma$ is $\infty$-central, and therefore lies in the image of Hochschild cohomology. 
	\end{prop}
	\begin{proof} By the combinatorial criterion of Proposition \ref{combinatorialcriterion}, it's enough to show the equality 
		\begin{align}
			\chi_\upgamma c_1 \dots c_n = c_1 \dots c_n \chi_\upgamma 
		\end{align}
		in $kC^\vee_N$ for all tuples $c_1, c_2, \dots, c_n$ with $c_j \in C_{k_j}^\vee$ and $N = (s-1) + k_1 + \dots + k_n + 1$. Let us write $c_j = b_j^\vee$ for the dual of a chain $b_j \in C_{k_j}$. Since $\upgamma_i^\vee b_j^\vee = (b_j \upgamma_i)^\vee$, we let $\chi^{\vee}_\upgamma = \sum_{i=0}^{n_{\upgamma} - 1} \upgamma_i$ in $kC_{s-1}$ and prove the dual equality 
			\begin{align*}
			\chi_\upgamma^\vee b_n \dots b_1 = b_n \dots b_1 \chi_\upgamma^\vee
			\end{align*}
			in $kC_{N}$. Let $\alpha = b_n \dots b_1$, considered as a path in $Q$. 
	
			If $\chi_\upgamma^\vee \alpha \neq 0$ in $kC_N$ then $\upgamma_i \alpha$ is an $N$-chain for some $i$, in which case $i$ is unique by rigidity (Lemma \ref{rigidity}). Writing $\upgamma_i = p_0 p_1 \dots p_{s-1}$ for the perfect walk with $t_{\upgamma_i} = p_{s-1}$, the unique extension property gives $\alpha = p_s \dots p_N$. We then rotate the string $\upgamma_i \alpha$: 
			\begin{align*}
				\upgamma_i \alpha &= (p_0 p_1 \dots p_{s-1}) p_s \dots p_N \\ 
				 &= (p_0 p_1 \dots p_{s-1}) p_0 \dots p_{N-s} \\ 
				 &= p_0 p_1 \dots p_{N-s}(p_{N-s + 1} \dots \dots p_{N-s}) \\ 
				 &= \alpha (p_{N-s + 1} \dots \dots p_{N-s}). 
			\end{align*}
By completeness (Lemma \ref{completeness}) we have $p_{N-s+1} \dots p_{N-s} = \upgamma_j$ for some $\upgamma_j \in \{ \upgamma_0, \upgamma_1, \dots, \upgamma_{n_{\upgamma} - 1} \}$, and as $\alpha \upgamma_j = \upgamma_i \alpha \neq 0$ in $kC_N$, rigidity again gives $\upgamma_j$ as the unique stable relation cycle with $\alpha \upgamma_j \neq 0$ in $kC_N$. Putting this together gives
			\begin{align*}
				\chi_\upgamma^\vee \alpha = \upgamma_i \alpha = \alpha \upgamma_j = \alpha \chi_{\upgamma}^\vee	
			\end{align*}
			in $kC_N$. 
More formally, this argument shows that $\chi_\upgamma^\vee \alpha \neq 0$ in $kC_N$ implies $\alpha \chi_\upgamma^\vee \neq 0$ in $kC_N$, in which case $\chi_\upgamma^\vee \alpha = \alpha \chi_\upgamma^\vee$ in $kC_N$. As the argument is symmetrical (using the dual properties established instead), we also see that $\alpha \chi_\upgamma^\vee \neq 0$ in $kC_N$ implies $\chi_\upgamma^\vee \alpha \neq 0$ in $kC_N$, in which case $\alpha \chi_\upgamma^\vee = \chi_\upgamma^\vee \alpha$ in $kC_N$, and so the claim holds.  
\end{proof}

\begin{rem} We thus see that periodicity operators lift to Hochschild cohomology, and any two lifts differ by an element of ${\rm ker}(\varphi_\Bbbk)$, which is nilpotent. However our proof shows the stronger statement that all operators $\chi_\upgamma \in {\rm Ext}^*_\Lambda(\Bbbk, \Bbbk)$ have an unambiguous lift to an operator $\chi_\upgamma \in {\rm HH}^*(\Lambda, \Lambda)$, which we denote by the same letter by abuse of notation. 

	This lift is described as follows \cite{BG}: the invariance of Hochschild cohomology under Koszul duality gives an isomorphism ${\rm HH}^*(\Lambda, \Lambda) \cong {\rm HH}^*(E, E)$, with $E = {\rm Ext}^*_\Lambda(\Bbbk, \Bbbk)$ treated as an $A_\infty$-algebra. Consider the cochain in $C^*(E, E)$ given by $1 \mapsto \chi_\upgamma$ and sending everything else to zero. The condition ${\rm ad}_{\chi_\upgamma} = 0$ is precisely the Hochschild cocycle condition and so this cocycle gives rise to a cohomology class $\chi_\upgamma \in {\rm HH}^*(E, E) \cong {\rm HH^*}(\Lambda, \Lambda)$, which provides a section for the characteristic morphism. Hence to any stable relation cycle $\upgamma$ we attach a well-defined class $\chi_\upgamma \in {\rm HH}^*(\Lambda, \Lambda)$. 
\end{rem} 


\subsection{The ring of periodicity operators}


We now turn to the study of the multiplicative properties of the periodicity operators $\{ \chi_\upgamma \}_\upgamma \subseteq {\rm Ext}^*_\Lambda(\Bbbk, \Bbbk)$. In this subsection we assume that $\Lambda$ is a Gorenstein monomial algebra of infinite global dimension. In this case $\Lambda$ always has stable relation cycles by Corollary \ref{perfpathstabilise}. 

		Let $\mathcal{R} := k\langle \chi_\upgamma \ | \ \upgamma \rangle \subseteq {\rm Ext}^*_\Lambda(\Bbbk, \Bbbk)$ be the $A_\infty$-central graded subalgebra generated by the classes $\{ \chi_\upgamma \}_{\upgamma}$. It will turn out that the graded connected algebra $\mathcal{R}$ is Noetherian, reduced and of Krull dimension one, and the structure theory of such rings dictates that $\mathcal{R}$ embeds in a finite product of polynomial algebras 
\begin{align}
	\mathcal{R} \subseteq \prod_{i=1}^b k[t_i]
\end{align}
with $k[t_i]$ the graded normalisation of $\mathcal{R}/\mathfrak{p}_i$ for $\mathfrak{p}_i$ the $i$-th minimal homogeneous prime. Composing with the $i$-th projection gives a map $\pi_i: \mathcal{R} \to k[t_i]$, and so for each $\chi_\upgamma \in \mathcal{R}$ we may write $\pi_i(\chi_\upgamma) = t_i^{n_i}$ for some $n_i \in \N \cup \{ -\infty \}$ (where we interpret $t_i^{-\infty} = 0$). Denoting this exponent by $n_i = {\sf ord}_i(\upgamma)$, it is clear that the ring structure of $\mathcal{R}$ is determined by the values $b$ and ${\sf ord}_i(\upgamma)$ as $i = 1, 2, \dots, b$ and $\upgamma$ ranges over all stable relation cycles.  Turning this around, we will first give a combinatorial definition of $b$ and ${\sf ord}_i(\upgamma)$ which we will use to establish the structure of $\mathcal{R}$ as claimed in the above paragraph. In fact we will obtain a slightly more precise description of the structure of $\mathcal{R}$, see Prop. \ref{ringstructure} below.

\subsubsection*{Branch equivalence relation} Let $\upgamma_1, \upgamma_2$ be stable relation cycles, and recall that $\upgamma_1 \sim \upgamma_2$ if they have the same associated set (in particular then $\chi_{\upgamma_1} = \chi_{\upgamma_2}$). It is easy to see that powers of stable relation cycles are also stable relation cycles and that $\upgamma_1 \sim \upgamma_2$ implies $\upgamma_1^n \sim \upgamma_2^n$ for any $n \geq 2$, and so this equivalence relation determines another coarser equivalence relation. 
\begin{defn}[Branch equivalence relation] The stable relation cycles $\upgamma_1, \upgamma_2$ are branch equivalent if there exists $n_1, n_2 \geq 1$ such that $\upgamma_1^{n_1} \sim \upgamma_2^{n_2}$. We denote the branch equivalence relation by $\upgamma_1 \approx \upgamma_2$. 
\end{defn}

\begin{lem}\label{srclemma0} Let $\upgamma_1, \upgamma_2$ be stable relation cycles. If $t_{\upgamma_1} = t_{\upgamma_2}$ or $h_{\upgamma_1} = h_{\upgamma_2}$, then $\upgamma_1 \approx \upgamma_2$. 
\end{lem}
\begin{proof} We assume that $t_{\upgamma_1} = t_{\upgamma_2} = p$, the other case is dual. We can write 
\begin{align}
	\upgamma_1 &= p_0 p_1 \dots p_{s-1} \\ 
	\upgamma_2 &= p'_0 p'_1 \dots p'_{s'-1} 
\end{align}
for perfect cycles with $p_{s-1} = p'_{s'-1} = p$. This implies $p_0 = p'_0$, $p_1 = p'_1$, \dots, and so on, so that ${\bf p} = (p_0, p_1, \dots, p_{s-1})$ and ${\bf p'} = (p'_0, p'_1, \dots, p'_{s'-1})$ are both powers of the same perfect cycle of minimal length $s'' \mid s, s'$. Taking a common power, we see that there are $n_1, n_2 \geq 1$ such that $\upgamma_1^{n_1} = \upgamma_2^{n_2}$, and so $\upgamma_1 \approx \upgamma_2$.
\end{proof}

\begin{lem} There are finitely many branch equivalence classes $\{ \Gamma_i \}_{i \in I}$ of stable relation cycles.
\end{lem}
\begin{proof}
	The map $\bigcup_{i \in I} \Gamma_i \to \{ \textup{perfect paths} \}$ sending $\upgamma$ to its tail $t_\upgamma$ sends distinct equivalence classes into disjoint non-empty subsets by Lemma \ref{srclemma0}, and so $|I|$ is bounded above by the number of perfect paths. 
\end{proof}

It follows that we can identify $I = \{ 1,2, \dots, b \}$ and so list the branch equivalence classes as $\Gamma_1, \Gamma_2, \dots, \Gamma_b$ for some $b \in \N$. We call $b$ the number of branches. 

Fixing an equivalence class $\Gamma_i$, consider the set of path lengths of elements of $\Gamma_i$ which we denote $\widetilde{\N_{\Gamma_i}} = \{ {\sf len}(\upgamma) \ | \ \upgamma \in \Gamma_i \} \subseteq \N$. Letting $\gcd_i := \gcd(\widetilde{\N_{\Gamma_i}})$, we normalise this to 
\begin{align}
	\N_{\Gamma_i} := \frac{1}{\gcd_i} \widetilde{\N_{\Gamma_i}} \subseteq \N. 
\end{align}

\begin{defn} We define the function ${\sf ord}_i: \{ \upgamma \} \to \N \cup \{ - \infty \}$ on the set of stable relation cycles by 
	\begin{align*}
		{\sf ord}_i(\upgamma) :=	\begin{cases} \frac{1}{\gcd_i} {\sf len}(\upgamma) & \textup{ if } \upgamma \in \Gamma_i \\  -\infty & \textup{ if } \upgamma \notin \Gamma_i. \end{cases} 
	\end{align*} 
\end{defn}

Note in particular that $\N_{\Gamma_i} = \{ {\sf ord}_i(\upgamma) \ | \ \upgamma \in \Gamma_i \} \subseteq \N$.

\begin{lem}\label{srclemma1} Let $\upgamma_1, \upgamma_2 \in \Gamma_i$. Then $\upgamma_1 \sim \upgamma_2$ if and only if ${\sf ord}_i(\upgamma_1) = {\sf ord}_i(\upgamma_2)$. 
\end{lem}
\begin{proof}
	Equivalently we show that $\upgamma_1 \sim \upgamma_2$ if and only if ${\sf len}(\upgamma_1) = {\sf len}(\upgamma_2)$. Since stable relation cycles in the same associated set have the same path length, the necessary implication holds by definition. For the converse, assume that ${\sf len}(\upgamma_1) = {\sf len}(\upgamma_2) = n$. Write $\upgamma_1 = a_0 a_1 \dots a_{n-1}$ and $\upgamma_2 = a'_0 a'_1 \dots a'_{n-1}$ with $a_i, a'_i \in Q_1$. Since $\upgamma_1 \approx \upgamma_2$, there are $n_1, n_2 \geq 1$ such that $\upgamma_1^{n_1} \sim \upgamma_2^{n_2}$, and so the two infinite periodic strings 
	\begin{align}
		&\underline{a}\ : 	\quad  \dots (a_0 a_1 \dots a_{n-1})^{n_1} (a_0 a_1 \dots a_{n-1})^{n_1} (a_0 a_1 \dots a_{n-1})^{n_1} \dots \\ 
		&\underline{a}': 	\quad \dots (a'_0 a'_1 \dots a'_{n-1})^{n_2} (a'_0 a'_1 \dots a'_{n-1})^{n_2} (a'_0 a'_1 \dots a'_{n-1})^{n_2} \dots 
	\end{align}
	must coincide up to some translation. But then $\upgamma_1, \upgamma_2$ are two length $n$ substrings of the same infinite periodic string 
	\begin{align}
		\dots a_0 a_1\dots a_{n-1} (a_0 a_1 \dots a_{n-1}) a_0 a_1 \dots a_{n-1} \dots 		
	\end{align} and so $\upgamma_1 \sim \upgamma_2$ by contruction of the associated set. 
\end{proof} 

\subsubsection*{Multiplicative properties} We next establish the multiplicative properties of $\{ \chi_\upgamma \}_{\upgamma}$. We aim to prove the following: 

\begin{prop*}[Prop. \ref{multiplicativeproperties}] Let $\upgamma$ and $\upgamma_1, \upgamma_2, \upgamma_3$ denote stable relation cycles. 
	\begin{enumerate}[i)]
	\item If $\chi_{\upgamma_1} \cdot \chi_{\upgamma_2} \neq 0$, then $\upgamma_1 \approx \upgamma_2$. 
	\item Assume that $\upgamma_1 \approx \upgamma_2$. Then there is a $\upgamma_3 \approx \upgamma_1, \upgamma_2$ such that $\chi_{\upgamma_1} \cdot \chi_{\upgamma_2} = \chi_{\upgamma_3}$, and then ${\sf ord}_i(\upgamma_1) + {\sf ord}_i(\upgamma_2) = {\sf ord}_i(\upgamma_3)$. 
	\item We have $\chi_\upgamma^n = \chi_{\upgamma^n}$ for any $n \geq 1$. 
\end{enumerate}
\end{prop*} 

The proof will be done in a series of steps. Let us first collect some comments. 

By property i) it is enough to understand the product of $\chi_{\upgamma_1}, \chi_{\upgamma_2}$ for $\upgamma_1, \upgamma_2 \in \Gamma_i$ within the same branch equivalence class. Property ii) then tells us that $\{ \chi_{\upgamma} \}_{\upgamma \in \Gamma_i }$ forms a multiplicative basis for the subalgebra $\mathcal{R}_i := k\langle \chi_\upgamma \ | \ \upgamma \in \Gamma_i \rangle \subseteq \mathcal{R}$. To analyse the ring structure, we may extend the notation ${\sf ord}_i(\chi_\upgamma) := {\sf ord}_i(\upgamma)$ to the class $\chi_\upgamma$, which is independent of choice of representative $\upgamma$ by Lemma \ref{srclemma1}, and in fact Lemma \ref{srclemma1} tells us that any value $n \in \N_{\Gamma_i}$ can be represented by a unique class $\chi_{\upgamma}$ for $\upgamma \in \Gamma_i$. Hence the function ${\sf ord}_i$ sets up a bijection $\{ \chi_{\upgamma} \}_{\upgamma \in \Gamma_i} \leftrightarrow \N_{\Gamma_i}$. Since ${\sf ord}_i(-)$ is additive on products, that the set $\{ \chi_{\upgamma} \}_{\upgamma \in \Gamma_i }$ forms a multiplicative basis for $\mathcal{R}_i$ by property ii) translates into $\N_{\Gamma_i} \subseteq \N$ being a subsemigroup, and we see that the above bijection induces an algebra isomorphism onto the semigroup algebra 
\begin{align}
\mathcal{R}_i = k\langle \chi_{\upgamma} \ | \ \upgamma \in \Gamma_i \rangle	\xrightarrow{\cong} k[t^{\N_{\Gamma_i}}] := k\langle t^n \ | \ n \in \N_{\Gamma_i} \rangle \subseteq k[t]
\end{align}
sending $\chi_\upgamma \mapsto t^{{\sf ord}_i(\upgamma)}$. 

Keeping this in mind, we see that many conditions on the behavior of ${\sf ord}_i(-)$ and $\{ \chi_\upgamma\}_{\upgamma}$ have to be verified. First, assuming the above isomorphism we see that the ring $\mathcal{R}_i \cong k[t^{\N_{\Gamma_i}}]$ is bigraded by ${\sf ord}_i(\chi_\upgamma)$ and cohomological degree $|\chi_\upgamma|$, and either degree determines the other. This translates into the next simple combinatorial lemma. 

If $\upgamma$ is an $n$-chain, let us call $n$ the Anick degree of $\upgamma$. 
\begin{lem}\label{degreeandlength} Let $\upgamma_1, \upgamma_2 \in \Gamma_i$. Then $\upgamma_1, \upgamma_2$ have the same Anick degree if and only if $\upgamma_1 \sim \upgamma_2$. 
\end{lem}
\begin{proof} If $\upgamma_1 \sim \upgamma_2$ then the implication holds by definition. Conversely assume that $\upgamma_1, \upgamma_2$ are both $n$-chains. Since $\upgamma_1 \approx \upgamma_2$, there exists $n_1, n_2 \geq 1$ such that $\upgamma_1^{n_1} \sim \upgamma_2^{n_2}$. The relation $\sim$ preserves both the Anick degree and the path length, and so we have $n_1 n = n_2 n$ and $n_1{\sf len}(\upgamma_1) = n_2 {\sf len}(\upgamma_2)$. Cancellation first gives $n_1 = n_2$ and then ${\sf len}(\upgamma_1) = {\sf len}(\upgamma_2)$. Thus ${\sf ord}_i(\upgamma_1) = {\sf ord}_i(\upgamma_2)$ and so $\upgamma_1 \sim \upgamma_2$ by Lemma \ref{srclemma1}. 
\end{proof} 

Next, understanding the product $\chi_{\upgamma_1} \cdot \chi_{\upgamma_2}$ will quickly reduce to understanding triples of stable relation cycles $\upgamma_1, \upgamma_2, \upgamma_3$ such that $s(\upgamma_2) = t(\upgamma_1)$ and $\upgamma_3 = \upgamma_2 \upgamma_1$. In particular establishing identities of the form $\chi_{\upgamma_1} \cdot \chi_{\upgamma_2} = \chi_{\upgamma_3}$ will require understanding the behavior of associated sets under products of stable relation cycles. This is done in the next two lemma. 

\begin{lem}[Product lemma]\label{productlemma} Let $\upgamma_1, \upgamma_2$ be stable relation cycles of Anick degrees $s_1-1$ and $s_2-1$, respectively. Assume that $s(\upgamma_2) = t(\upgamma_1)$ and $\upgamma_2 \upgamma_1$ is also a stable relation cycle. 
	\begin{enumerate}[i)]
	\item We have $\upgamma_1 \approx \upgamma_2 \approx \upgamma_2\upgamma_1$. 
\item If $\widetilde{\upgamma}_2 \sim \upgamma_2$ and $\widetilde{\upgamma}_1 \sim \upgamma_2$ are equivalent stable relation cycles with $s(\widetilde{\upgamma}_2) = t(\widetilde{\upgamma}_1)$ and $\widetilde{\upgamma}_2 \widetilde{\upgamma}_1$ also a stable relation cycle, then $\widetilde{\upgamma}_2 \widetilde{\upgamma}_1 \sim \upgamma_2 \upgamma_1$. 
\end{enumerate}
\end{lem}
\begin{proof}
i). Since $\upgamma_1, \upgamma_2$ and $\upgamma_2 \upgamma_1$ are stable relation cycles, the unique extension property (Prop. \ref{uniqueextensionproperty}) shows that $t_{\upgamma_2 \upgamma_1} = t_{\upgamma_2}$ and $h_{\upgamma_2 \upgamma_1} = h_{\upgamma_1}$. Lemma \ref{srclemma0} then gives $\upgamma_2 \approx \upgamma_2 \upgamma_1 \approx \upgamma_1$. 

ii). Applying i) to $\widetilde{\upgamma}_2, \widetilde{\upgamma}_1$ gives $\widetilde{\upgamma}_2 \widetilde{\upgamma}_1 \approx \widetilde{\upgamma}_1 \sim \upgamma_1 \approx \upgamma_2 \upgamma_1$, and so $\widetilde{\upgamma}_2 \widetilde{\upgamma}_1 \approx \upgamma_2 \upgamma_1$. Since the stable relation cycles $\widetilde{\upgamma}_2 \widetilde{\upgamma}_1$ and $\upgamma_2 \upgamma_1$ both have Anick degree $s_3-1$ for $s_3 = s_2 + s_1$, Lemma \ref{degreeandlength} shows that $\widetilde{\upgamma}_2 \widetilde{\upgamma}_1 \sim \upgamma_2 \upgamma_1$. 
\end{proof}

\begin{lem}[Product decompositions for associated sets]\label{productdecomposition} Let $\upgamma, \upgamma_1, \upgamma_2, \upgamma_3$ denote stable relation cycles and assume that $\upgamma_3 = \upgamma_2 \upgamma_1$. Let $S_i = \{ \widetilde{\upgamma}_i \ | \ \widetilde{\upgamma}_i \sim \upgamma_i \}$ be the associated sets of $\upgamma_i$, $i= 1,2,3$. 
	\begin{enumerate}[i)] 
	\item The associated set $S_3$ decomposes as a product of elements of $S_2, S_1$, in that: 
\begin{enumerate}[a)] 
\item For every $\widetilde{\upgamma}_2 \in S_2$, there is a unique $\widetilde{\upgamma}_1 \in S_1$ such that $\widetilde{\upgamma}_2 \widetilde{\upgamma}_1 \in S_3$. 
\item For every $\widetilde{\upgamma}_1 \in S_1$, there is a unique $\widetilde{\upgamma}_2 \in S_2$ such that $\widetilde{\upgamma}_2 \widetilde{\upgamma}_1 \in S_3$. 
\item Every $\widetilde{\upgamma}_3 \in S_3$ is of the form $\widetilde{\upgamma}_3 = \widetilde{\upgamma}_2 \widetilde{\upgamma}_1$ for unique $\widetilde{\upgamma}_2 \in S_2$, $\widetilde{\upgamma}_1 \in S_1$. 
	\end{enumerate}\medskip

\noindent Hence we have $S_3 = S_2 \cdot S_1 := \{ \widetilde{\upgamma}_2 \widetilde{\upgamma}_1 \ | \ \widetilde{\upgamma}_i \in S_i \textup{ and } (\widetilde{\upgamma}_2, \widetilde{\upgamma}_1) \textup{ compatible} \}$, with the compatibility in the sense of a)-b). \medskip

\item For any $n \geq 1$ the associated set of $\upgamma^n$ is $\{ \widetilde{\upgamma}^n \ | \ \widetilde{\upgamma} \sim \upgamma \}$. 
\end{enumerate}

\end{lem}
\begin{proof}
i). Let $s_i - 1$ be the Anick degree of $\upgamma_i$, so that $s_3 = s_2 + s_1$. We prove 1). Let $\widetilde{\upgamma}_2 \in S_2$ and write $\widetilde{\upgamma}_2 = p_0 p_1 \dots p_{s_2 -1}$ for the perfect cycle with $t_{\widetilde{\upgamma}_2} = p_{s_2-1}$, and extend this to a perfect walk $p_0 p_1 \dots p_{s_2 - 1} p_0 p_1 \dots p_{s_1 - 1}$ of length $s_2 + s_1 = s_3$. Set $\widetilde{\upgamma}_1 := p_0 p_1 \dots p_{s_1 - 1}$ and write $\widetilde{\upgamma}_3 := \widetilde{\upgamma}_2 \widetilde{\upgamma}_1$. Note that $\widetilde{\upgamma}_1, \widetilde{\upgamma}_3$ are given by perfect walks of appropriate length and so are stable relation cycles by the recognition lemma (Lemma \ref{recognitionlemma}). 

Part i) of the product lemma (Lemma \ref{productlemma}) gives $\widetilde{\upgamma}_1 \approx \widetilde{\upgamma}_2 \sim \upgamma_2 \approx \upgamma_1$, and so $\widetilde{\upgamma}_1 \approx \upgamma_1$. Since $\widetilde{\upgamma}_1$ and $\upgamma_1$ are both $(s_1-1)$-chains, Lemma \ref{degreeandlength} then gives $\widetilde{\upgamma}_1 \sim \upgamma_1$, and finally part ii) gives $\widetilde{\upgamma}_3 = \widetilde{\upgamma}_2 \widetilde{\upgamma}_1 \sim \upgamma_2 \upgamma_1 = \upgamma_3$. Finally, the uniqueness follows from the unique extension property (Prop. \ref{uniqueextensionproperty}), and 1) follows. The proof of 2) is dual. 

For 3), let $\widetilde{\upgamma}_3 \in S_3$ and write $\widetilde{\upgamma}_3 = p_0 p_1 \dots p_{s_3 - 1}$ for the perfect cycle with tail $t_{\widetilde{\upgamma}_3} = p_{s_3 - 1}$. Since $s_3 = s_2 + s_1$ and both $s_1, s_2$ are multiples of $\ell$, and therefore of the minimal period of this perfect cycle, we can break it down further as $\widetilde{\upgamma}_3 = p_0 p_1 \dots p_{s_3 - 1} = (p_0 p_1 \dots p_{s_1 - 1})(p_0 p_1 \dots p_{s_2 - 1})$. Letting $\widetilde{\upgamma}_i := p_0 p_1 \dots p_{s_i - 1}$, since the length $s_i$ is appropriate the recognition lemma (Lemma \ref{recognitionlemma}) again shows that $\widetilde{\upgamma}_i$ are stable relation cycles. Part i) of the product lemma (Lemma \ref{productlemma}) then gives $\widetilde{\upgamma}_i \approx \widetilde{\upgamma}_3 \sim \upgamma_3 \approx \upgamma_i$ for $i = 1,2$, so that $\widetilde{\upgamma}_i \approx \upgamma_i$, and Lemma \ref{degreeandlength} then gives $\widetilde{\upgamma}_i \sim \upgamma_i$. Hence $\widetilde{\upgamma}_3 = \widetilde{\upgamma}_2 \widetilde{\upgamma}_1$ with $\widetilde{\upgamma}_2 \in S_2$, $\widetilde{\upgamma}_1 \in S_1$ as claimed. Finally, to see uniqueness assume that $\widetilde{\upgamma}_3 = \widetilde{\upgamma}_2 \widetilde{\upgamma}_1 = \widetilde{\upgamma}'_2 \widetilde{\upgamma}_1'$ for some possibly different $\widetilde{\upgamma}'_1 \in S_1, \widetilde{\upgamma}'_2 \in S_2$. The unique extension property (Prop. \ref{uniqueextensionproperty}) implies that $t_{\widetilde{\upgamma}_1} = t_{\widetilde{\upgamma}_3} = t_{\widetilde{\upgamma}'_1}$. Lemma \ref{srclemma0} then gives $\widetilde{\upgamma}'_1 \approx \widetilde{\upgamma}_1$, and since they have the same Anick degree Lemma \ref{degreeandlength} gives $\widetilde{\upgamma}'_1 \sim \widetilde{\upgamma}_1$. In particular ${\sf len}(\widetilde{\upgamma}'_1) = {\sf len}(\widetilde{\upgamma}_1)$, and so $\widetilde{\upgamma}'_1 = \widetilde{\upgamma}_1$. The equality $\widetilde{\upgamma}'_2 = \widetilde{\upgamma}_2$ then follows from 2), and we have shown 3).

ii). This follows from i) by setting $\upgamma_1 = \upgamma$ and $\upgamma_2 = \upgamma^{n-1}$ and induction on $n$.  
\end{proof}

We can now prove the claim. 
\begin{prop}\label{multiplicativeproperties} Let $\upgamma$ and $\upgamma_1, \upgamma_2, \upgamma_3$ denote stable relation cycles. 
	\begin{enumerate}[i)]
	\item If $\chi_{\upgamma_1} \cdot \chi_{\upgamma_2} \neq 0$, then $\upgamma_1 \approx \upgamma_2$. 
	\item Assume that $\upgamma_1 \approx \upgamma_2$. Then there is a $\upgamma_3 \approx \upgamma_1, \upgamma_2$ such that $\chi_{\upgamma_1} \cdot \chi_{\upgamma_2} = \chi_{\upgamma_3}$, and then ${\sf ord}_i(\upgamma_1) + {\sf ord}_i(\upgamma_2) = {\sf ord}_i(\upgamma_3)$. 
	\item We have $\chi_\upgamma^n = \chi_{\upgamma^n}$ for any $n \geq 1$. 
\end{enumerate}
\end{prop} 
\begin{proof}

i). If $\chi_{\upgamma_1} \cdot \chi_{\upgamma_2} \neq 0$ then there exists $\widetilde{\upgamma}_1 \sim \upgamma_1$ such that $\widetilde{\upgamma}_1^\vee \cdot \upgamma_2^\vee = (\upgamma_2 \widetilde{\upgamma}_1)^\vee \neq 0$. In this case the Anick chain $\upgamma_2 \widetilde{\upgamma}_1$ must be given by a perfect walk by the unique extension property (Prop. \ref{uniqueextensionproperty}), and $\upgamma_2 \widetilde{\upgamma}_1$ is then a stable relation cycle by the recognition lemma (Lemma \ref{recognitionlemma}) as this perfect walk has the required length. The product lemma (Lemma \ref{productlemma}) then gives $\upgamma_1 \sim \widetilde{\upgamma}_1 \approx \upgamma_2$ and so $\upgamma_1 \approx \upgamma_2$. 

\medskip

ii). Let $n_1, n_2 \geq 1$ be such that $\upgamma_1^{n_1} \sim \upgamma_2^{n_2}$. Then $\chi_{\upgamma_1^{n_1}} = \chi_{\upgamma_2^{n_2}}$ and so 
\begin{align}
	\chi_{\upgamma_1}^{n_1} \cdot \chi_{\upgamma_2}^{n_2} = 	
	\chi_{\upgamma_1^{n_1}} \cdot \chi_{\upgamma_2^{n_2}}
	= \chi_{\upgamma_1^{n_1}} \cdot \chi_{\upgamma_1^{n_1}}
	= 
	\chi_{\upgamma_1^{n_1 + n_1}} \neq 0. 
\end{align}
Hence $\chi_{\upgamma_1} \cdot \chi_{\upgamma_2} \neq 0$. Possibly replacing $\upgamma_1$ by $\widetilde{\upgamma}_1 \sim \upgamma_1$, we can assume that $\upgamma_1^{\vee} \upgamma_2^{\vee} = (\upgamma_2 \upgamma_1)^{\vee} \neq 0$; in particular $\upgamma_2 \upgamma_1$ is an Anick chain, and the recognition lemma again shows that $\upgamma_3 := \upgamma_2 \upgamma_1$ is a stable relation cycle. We are then in the setting of Lemma \ref{productdecomposition}.

The decomposition of the associated set $S_3 = S_2 \cdot S_1$ of Lemma \ref{productdecomposition} i) then gives 
\begin{align}
	\chi_{\upgamma_3} = \sum_{ \widetilde{\upgamma}_3 \in S_3 } \widetilde{\upgamma}_3^\vee 
	&= \sum_{ \widetilde{\upgamma}_2 \widetilde{\upgamma}_1 \in S_3 } (\widetilde{\upgamma}_2\widetilde{\upgamma}_1)^\vee \\ 
	&= \sum_{ \widetilde{\upgamma}_2 \widetilde{\upgamma}_1 \in S_3, } \widetilde{\upgamma}_1^\vee \widetilde{\upgamma}_2^\vee \\ 
	&= \sum_{ \widetilde{\upgamma}_1 \in S_1 } \widetilde{\upgamma}_1^\vee \cdot  \sum_{\widetilde{\upgamma}_2 \in S_2} \widetilde{\upgamma}_2^\vee \\ 
	&= \chi_{\upgamma_1} \cdot \chi_{\upgamma_2}
\end{align}
where the second equality writes $\widetilde{\upgamma}_3 = \widetilde{\upgamma}_2 \widetilde{\upgamma}_1$ in terms of the unique decomposition of Lemma \ref{productdecomposition} and the fourth equality follows from the Rigidity Lemma (Lemma \ref{rigidity}). This proves the main claim, and ${\sf ord}_i(\upgamma_1) + {\sf ord}_i(\upgamma_2) = {\sf ord}_i(\upgamma_3)$ simply follows from additivity of path lengths. 

\medskip

iii). For $n = 1$ there is nothing to prove and $n \geq 2$ follows from specialising the argument of ii) to $\upgamma_1 = \upgamma^{n-1}$ and $\upgamma_2 = \upgamma$, giving $\upgamma_3 = \upgamma^n$.  
\end{proof}

We can finally describe the ring structure of $\mathcal{R} = k\langle \chi_\upgamma \ | \ \upgamma \rangle \subseteq {\rm Ext}^*_\Lambda(\Bbbk, \Bbbk)$. First, we observed at the beginning of this subsection that Prop. \ref{multiplicativeproperties} implies that $\N_{\Gamma_i} \subseteq \N$ is a subsemigroup for each branch equivalence class $\Gamma_i$. 
Since $\gcd (\N_{\Gamma_i}) = 1$ by construction, we see that $|\N \setminus \N_{\Gamma_i}| < \infty$ by \cite[Lemma 2.1]{RGS}, so that $\N_{\Gamma_i}$ is a numerical semigroup. More importantly for us, $\N_{\Gamma_i}$ is a finitely generated semigroup \cite[Corollary 2.8]{RGS}. It follows that the (semigroup) algebra 
\begin{align} 
	k[t_i^{\N_{\Gamma_i}}] = k\langle t_i^n \ | \ n \in \N_{\Gamma_i} \rangle \subseteq k[t_i]. 
\end{align} 
is a finitely generated $k$-algebra, and in particular is Noetherian. 

Let us write $\varepsilon: k[t_i^{\N_{\Gamma_i}}] \to k$ for the augmentation with $\varepsilon(t_i^n) = 0$ for all $n \in \N_{\Gamma_i}$. Lastly, we grade the algebra above by setting $|t_i| = \gcd\{ | \chi_\upgamma | \ | \ \upgamma \in \Gamma_i \}$. 

\begin{prop}\label{ringstructure} The graded algebra $\mathcal{R} = k\langle \chi_{\upgamma} \ | \ \upgamma \rangle$ is Noetherian, reduced and of Krull dimension one. There is an embedding of graded algebras 
	\begin{align*}
		\phi: \mathcal{R} \hookrightarrow \prod_{i=1}^b k[t_i] 
	\end{align*}
	inducing an isomorphism onto the fibre product $\mathcal{R} \cong k[t_1^{\N_{\Gamma_1}}]\  _\varepsilon\! \! \times_\varepsilon \ \! \dots \ \!  _\varepsilon\! \!\times_\varepsilon k[t_b^{\N_{\Gamma_b}}]$. 
\end{prop} 
\begin{proof}
	It is clear that the graded algebra $k[t_1^{\N_{\Gamma_1}}]\  _\varepsilon\! \! \times_\varepsilon \ \! \dots \ \!  _\varepsilon\! \!\times_\varepsilon k[t_b^{\N_{\Gamma_b}}]$ is Noetherian, reduced and of Krull dimension one, and so it suffices to prove the second claim. Define the $k$-algebra morphism $\phi: \mathcal{R} \to \prod_{i=1}^b k[t_i]$ on generators by 
	\[
		\phi(\chi_\upgamma) = (0, \dots, 0, t_i^{{\sf ord}_i(\upgamma)}, 0,  \dots, 0) \textup{ for } \upgamma \in \Gamma_i
	\] 
	where one may interpret the zero coefficients as $t_j^{{\sf ord}_j(\upgamma)} = t_j^{-\infty} = 0$ for $j \neq i$. 

	Then $\phi$ is a well-defined $k$-algebra morphism by Lemma \ref{srclemma1} and Prop. \ref{multiplicativeproperties}. Moreover Prop. \ref{multiplicativeproperties} shows that $\{ \chi_\upgamma \}_\upgamma$ is a multiplicative basis for $\mathcal{R}$, which $\phi$ sends bijectively onto a multiplicative basis for $k[t_1^{\N_{\Gamma_1}}]\  _\varepsilon\! \! \times_\varepsilon \ \! \dots \ \!  _\varepsilon\! \!\times_\varepsilon k[t_b^{\N_{\Gamma_b}}] \subseteq \prod_{i=1}^b k[t_i]$, and so $\phi$ induces an algebra isomorphism of $\mathcal{R}$ onto its image. 

	Finally, the grading on $\mathcal{R}$ induces a grading onto $k[t_1^{\N_{\Gamma_1}}]\  _\varepsilon\! \! \times_\varepsilon \ \! \dots \ \!  _\varepsilon\! \!\times_\varepsilon k[t_b^{\N_{\Gamma_b}}]$ by transport of structure, and since $\N_{\Gamma_i} \subseteq \N$ are numerical semigroups (in particular $\gcd(\N_{\Gamma_i}) = 1)$ this forces $|t_i| = \gcd \{ |\chi_\upgamma| \ | \ \upgamma \in \Gamma_i \}$. Since this is the grading we imposed on the larger algebra $\prod_{i=1}^b k[t_i]$, we see that $\phi$ was in fact an embedding of graded subalgebras and we are done.  
\end{proof}

\subsection{Computation of a ring of periodicity operators.} Let $\Lambda = kQ/I$ for the quiver given by an oriented cycle on seven vertices with arrows
$a,b,c,d,e,f,g$ and relations $abcd, bcde, def, efg, fgab, gabc$. This example was considered by Green, \begin{wrapfigure}{r}{0.34\textwidth}
\vspace{-1 em}
\begin{tikzpicture}[auto, scale = 0.57]
    \foreach \a in {1,2,3,4,5,6,7}
    {
        \node (u\a) at ({\a*51.42}:3){$\bullet$};
        \draw [latex-,
        		line width = 1.15 pt,
        		domain=\a*51.42+5:(\a+1)*51.42-5] plot ({3*cos(\x)}, {3*sin(\x)});}
    \node (a1) at ({1*51.42+35}:3.5){$g$};
    \node (a2) at ({2*51.42+35}:3.5){$a$};
    \node (a3) at ({3*51.42+35}:3.5){$b$};
    \node (a4) at ({4*51.42+35}:3.5){$c$};
    \node (a5) at ({5*51.42+35}:3.5){$d$};  
    \node (a6) at ({6*51.42+35}:3.5){$e$};  
    \node (a7) at ({7*51.42+35}:3.5){$f$};  
\end{tikzpicture}
\vspace{-1 em}
\end{wrapfigure}
Snashall and Solberg in \cite[Example 4.1, 7.5]{GSS06}. The algebra $\Lambda$ is Gorenstein of dimension $6$, and has a single branch $\Gamma$ with numerical semigroup $\N_\Gamma = \{ 2, 3, \dots \}$. We then have $\mathcal{R} \cong k[t^2, t^3] \subset k[t]$ with $|t| = 4$, with $t^2, t^3$ corresponding to $\chi_{\upgamma_2}, \chi_{\upgamma_3}$ where $\upgamma_2 = (abcdefg)^2$ and $\upgamma_3 = (abcdefg)^3$.  
In \cite[Example 7.5]{GSS06} the authors compute the Hochschild cohomology ring modulo nilpotents as ${\rm HH}^*(\Lambda, \Lambda)/\mathcal{N} \cong k[t]$ with $|t| = 4$, and the embedding above is precisely the natural map $\mathcal{R} \hookrightarrow {\rm HH}^*(\Lambda, \Lambda)/\mathcal{N}$, with agreement in all degrees past the Gorenstein dimension. Note that the ring $\mathcal{R}$ does not contain elements of lower degree by design.

\section{Main theorem} \label{sec:main}

Using the machinery developed in the previous sections we are now able to characterise the monomial algebras satisfying the ${\bf Fg}$ conditions of Snashall-Solberg. We begin with some minor setup.  

Consider a Gorenstein monomial algebra $\Lambda$, and we may assume that $\gldim \Lambda = \infty$ as all results below will immediately reduce to this case. To prove that ${\bf Fg}$ holds for $\Lambda$, it is enough to show that ${\rm Ext}^*_\Lambda(\Bbbk, \Bbbk)$ is module-finite over its $A_\infty$-central subalgebra $\mathcal{R} \subseteq {\rm Ext}^*_\Lambda(\Bbbk, \Bbbk)$. We will prove a slightly stronger statement with respect to a Noether normalisation of $\mathcal{R}$. 

Let $\Gamma_1, \Gamma_2, \dots, \Gamma_b$ be all branch equivalence classes of stable relation cycles, with accompanying numerical semigroups $\N_{\Gamma_1}, \N_{\Gamma_2}, \dots, \N_{\Gamma_b}$. We have seen in Prop. \ref{ringstructure} that $\mathcal{R}$ depends only on the structure of these numerical semigroups, and this makes it easy to construct a Noether normalisation.  

Let $n_i := \min \N_{\Gamma_i}$ and let $\upgamma_i \in \Gamma_i$ be a stable relation cycle with ${\sf ord}_i(\upgamma_i) = n_i$. Consider the class $\chi_i := \chi_{\upgamma_i}$, which is independent of the choice of $\upgamma_i$ above by Lemma \ref{srclemma1}. Introduce the following class $\chi \in \mathcal{R}$: 
\begin{equation}\label{chiclass} 
	\chi := \sum_{i=1}^b \chi_i^{m_i} 
\end{equation}
where $m_i \geq 1$ are such that $m_i |\chi_i| = m_j |\chi_j|$ for all $i,j$ and $\gcd \{ m_i \} = 1$. From Prop. \ref{ringstructure} it is clear that the polynomial subalgebra $k[\chi] \subseteq \mathcal{R}$ is a Noether normalisation, and we let $p := |\chi| = m_i |\chi_i|$ denote its degree. We note that $p$ is a multiple of the period $\ell$ of $\Lambda$. 

Next, recall that by Corollary \ref{perfpathstabilise}, any perfect path $w \in \Lambda$ over a Gorenstein monomial algebra can be extended to a perfect cycle $(w_0, w_1, \dots, w_{s-1})$ with $w = w_{s-1}$ such that $\upgamma = w_0 w_1 \dots w_{s-1}$ is a stable relation cycle. 
\begin{lem}\label{minvalue} Let $\Lambda$ be a Gorenstein monomial algebra and $w \in \Lambda$ be a perfect path. Let $\upgamma = w_0 w_1 \dots w_{s-1}$ be a stable relation cycle with $w = w_{s-1}$, and $\upgamma$ has minimal path length amongst stable relation cycles with this property. If $\upgamma \in \Gamma_i$, then ${\sf ord}_i(\upgamma) \in \N_{\Gamma_i}$ takes on the minimal value. 
\end{lem}
\begin{proof}
	Let $\upgamma_i \in \Gamma_i$ be a class with ${\sf ord}_i(\upgamma_i)$ minimal; we will show that $\upgamma \sim \upgamma_i$, so that ${\sf ord}_i(\upgamma) = {\sf ord}_i(\upgamma_i)$. Since $\upgamma \approx \upgamma_i$, there are $n, n_i \geq 1$ such that $\upgamma^n \sim \upgamma_i^{n_i}$. We have seen in Lemma \ref{productdecomposition} that the associated set of $\upgamma_i^{n_i}$ is of the form $\{ \widetilde{\upgamma}_i^{n_i} \ | \ \widetilde{\upgamma}_i \sim \upgamma_i \}$, and so $\upgamma^n = \widetilde{\upgamma}_i^{n_i}$ for some $\widetilde{\upgamma}_i \sim \upgamma_i$. The unique extension property (Prop. \ref{uniqueextensionproperty}) then gives $t_{\upgamma} = t_{\upgamma^n} = t_{\widetilde{\upgamma}_i^{n_i}} = t_{\widetilde{\upgamma}_i}$, and so $\upgamma, \widetilde{\upgamma}_i$ are stable relation cycles with the same tails. 

	Let $p = t_{\upgamma} = t_{\widetilde{\upgamma}_i}$ be the common tail, and note that we may have $p \neq w_{s-1}$ as we did not assume $w_{s-1}$ was the tail of $\upgamma = w_0 w_1 \dots w_{s-1}$. We can then rewrite the $(s-1)$-chain $\upgamma$ as $\upgamma = p_0 p_1 \dots p_{s-1}$ for the perfect cycle $(p_0, p_1, \dots, p_{s-1})$ ending in $p = p_{s-1}$. While we may have $w_i \neq p_i$ in general, note that the relations $w_0 w_1 = p_0 p_1$, \dots, $w_{s-2} w_{s-1} = p_{s-2} p_{s-1}$ are unique by Lemma \ref{uniquenessofrelations}, and therefore we have equalities of proper right subtrings $\upgamma_{2k}$ of $\upgamma$: 
	\begin{align}
		\upgamma_{2k}:= p_{2k} p_{2k+1} \dots p_{s-2} p_{s-1} = w_{2k} w_{2k+1} \dots w_{s-2} w_{s-1}.	
	\end{align}
	Our assumption on $\upgamma$ shows that $\upgamma_{2k}$ is not a stable relation cycle for any $k > 0$. It follows that $\upgamma$ also has the smallest path length amongst stable relation cycles with tail $t_{\upgamma} = p$. Since $\widetilde{\upgamma}_i$ has minimal value of ${\sf ord}_i(\widetilde{\upgamma}_i) = \frac{1}{\gcd_i} {\sf len}(\widetilde{\upgamma}_i)$ in $\N_{\Gamma_i}$ and tail $t_{\widetilde{\upgamma}_i} = p$, we conclude that $\upgamma = \widetilde{\upgamma}_i$ and so finally $\upgamma \sim \upgamma_i$. 
\end{proof}
	
By the lemma we see that any perfect path $w$ in $\Lambda$ can be extended to a stable relation cycle $\upgamma = w_0 w_1 \dots w_{s-1}$, not necessarily with $t_{\upgamma} = w_{s-1}$, such that $\chi_\upgamma = \chi_i$ agrees with a term of $\chi$ in \eqref{chiclass}, where $\upgamma \in \Gamma_i$ belongs to the $i$-th branch equivalence class. This will let us control multiplication against $\chi \in {\rm Ext}^*_\Lambda(\Bbbk, \Bbbk)$, as in the next proposition: 

\begin{prop}\label{chiperiodicity} Let $\Lambda$ be a Gorenstein monomial algebra and $\chi$ be the class defined in \eqref{chiclass}. Then left multiplication by $\chi$ 
	\begin{align*}
		\chi \cdot -: {\rm Ext}^n_\Lambda(\Bbbk, \Bbbk) \to {\rm Ext}^{n+p}_\Lambda(\Bbbk, \Bbbk) 
	\end{align*} 
	is an isomorphism for all $n \geq \dim \Lambda + 1$ and an epimorphism for $n = \dim \Lambda$. 
\end{prop} 
\begin{proof}  If $\gldim \Lambda < \infty$ the claim is trivial and so assume that $\gldim \Lambda = \infty$. We consider the action of $\chi \cdot -$ on the basis of chains $kC_{n-1}^\vee \cong {\rm Ext}^n_\Lambda(\Bbbk, \Bbbk)$ and prove that $\chi \cdot -$ is injective for $n \geq \dim \Lambda + 1$ and surjective for $n \geq \dim \Lambda$. 

	Let $n \geq \dim \Lambda + 1$. Then by Theorem \ref{gorensteincharacterisation} any $\upgamma_{n-1} \in C_{n-1}$ can be written as $\upgamma_{n-1} = \upgamma_{n-2} p$ for $p = t_{\upgamma_{n-1}}$ perfect (using $n > \dim \Lambda$). Writing $p = p_{s-1}$, we let ${\bf p} = (p_0, p_1, \dots, p_{s-1})$ be the perfect cycle of minimal length such that $\upgamma = p_0 p_1 \dots p_{s-1}$ is a stable relation cycle. Letting $\Gamma_i$ be the branch equivalence class of $\upgamma$, Lemma \ref{minvalue} shows that ${\sf ord}_i(\upgamma) \in \N_{\Gamma_i}$ attains the minimal value. In particular $\upgamma \sim \upgamma_i$ for any other $\upgamma_i \in \Gamma_i$ with ${\sf ord}_i(\upgamma_i)$ also attaining the minimum value.  

Writing $\chi = \sum_{i=1}^b \chi_{i}^{m_i}$ as in \eqref{chiclass} where $\chi_i := \chi_{\upgamma_i}$ for $\upgamma_i \in \Gamma_i$ a class with ${\sf ord}_i(\upgamma_i)$ minimal, this shows that each $(n-1)$-chain $\upgamma_{n-1}$ has tail $p$ given by a perfect path occuring as tail of a stable relation cycle $\upgamma$ with $\upgamma \sim \upgamma_i$ for some $i$. The unique extension property (Prop. \ref{uniqueextensionproperty}) and the rigiditiy lemma (Lemma \ref{rigidity}) then give 
\begin{align}
	\chi \cdot \upgamma_{n-1}^\vee &= \chi_{i}^{m_i} \cdot \upgamma_{n-1}^\vee \\ 
	&= (\upgamma^{m_i})^\vee \cdot \upgamma_{n-1}^\vee\\
	&= (\upgamma_{n-1} \upgamma^{m_i})^\vee 
\end{align}
where $\upgamma_{n-1} \upgamma^{m_i}$ is the unique right extension of $\upgamma_{n-1}$ as an $(n-1+p)$-chain. Since $\upgamma_{n-1}$ can be recovered from the $(n-1+p)$-chain $\upgamma_{n-1} \upgamma^{m_i}$ by successively removing tails, this shows that $\chi \cdot -: kC_{n-1}^\vee \to kC_{n-1+p}^\vee$ is injective. 

Now let $n \geq \dim \Lambda$. We prove that $\chi \cdot -: kC_{n-1}^\vee \to kC_{n-1+p}^\vee$ is surjective analogously. Corollary \ref{highdegreechains} shows that every $(n-1+p)$-chain is of the form 
\begin{align}
\upgamma_{n-1} w_n w_{n+1} \dots w_{n-1+p} 
\end{align}
for a perfect walk $(w_n, w_{n+1}, \dots, w_{n-1+p})$ and an $(n-1)$-chain $\upgamma_{n-1}$. Since the perfect walk $(w_n, w_{n+1}, \dots, w_{n-1+p})$ has length $p \geq \dim \Lambda + 1$, with $p$ even and a multiple of $\ell$, the recognition lemma (Lemma \ref{recognitionlemma}) shows that $w_n w_{n+1} \dots w_{n-1+p}$ is a stable relation cycle, and we may write $w_n w_{n+1} \dots w_{n-1+p} = \upgamma^m$ for some stable relation cycle $\upgamma$ of minimal path length and some $m \geq 1$. Letting $\Gamma_i$ again be the class of $\upgamma$, Lemma \ref{minvalue} shows that ${\sf ord}_i(\upgamma)$ takes on minimal value so that $\upgamma \sim \upgamma_i$; in particular $\chi_{\upgamma} = \chi_{\upgamma_i}$ and the equality $m |\chi_\upgamma| = p = m_i |\chi_{\upgamma_i}|$ forces $m = m_i$. It follows that every $(n-1+p)$-chain has the form $\upgamma_{n-1} \upgamma^{m_i}$, and so every basis element of $kC_{n-1+p}^\vee$ can be written as 
\begin{align}
	(\upgamma_{n-1} \upgamma^{m_i})^\vee = \chi \cdot \upgamma_{n-1}^\vee	
\end{align}
for some $\upgamma_{n-1}^\vee \in C_{n-1}^\vee$. This proves surjectivity, and we are done. 
\end{proof}

\begin{rem} The bound $n \geq \dim \Lambda + 1$ cannot be improved to $n \geq \dim \Lambda$ in general. This is clear if $\gldim \Lambda < \infty$ as $\chi = 0$ and ${\rm Ext}^{\dim \Lambda}_\Lambda(\Bbbk, \Bbbk) \neq 0$ then, but even when $\gldim \Lambda = \infty$ this can fail due to non-trivial projective summands in $\Omega^{\dim \Lambda}\ \! \Bbbk$ as seen in Section \ref{sec:gorensteinexample}. 
\end{rem}

We now obtain the main result of this paper. 
\begin{thm}\label{gorensteinfg} Let $\Lambda$ be a monomial algebra. Then $\Lambda$ satisfies ${\bf Fg}$ if and only if $\Lambda$ is Gorenstein. 
\end{thm}
\begin{proof} As we mentioned before, the necessary implication is due to Erdmann-Holloway-Snashall-Solberg-Taillefer in \cite[Th. 1.5]{EHSST}. Conversely assume that $\Lambda$ is Gorenstein, and that $\gldim \Lambda = \infty$ as the result is trivial otherwise. 

In this case Proposition \ref{chiperiodicity} shows that $\Ext^*_\Lambda(\Bbbk, \Bbbk)$, as a module over its polynomial subalgebra $k[\chi] \subseteq \mathcal{R} \subseteq \mathcal{Z}_\infty {\rm Ext}^*_\Lambda(\Bbbk, \Bbbk)$, is generated by elements of degree at most $\dim \Lambda + p - 1$. It follows that ${\rm Ext}^*_\Lambda(\Bbbk, \Bbbk)$ and $\mathcal{Z}_\infty = \mathcal{Z}_\infty {\rm Ext}^*_\Lambda(\Bbbk, \Bbbk) \subseteq {\rm Ext}^*_\Lambda(\Bbbk, \Bbbk)$ are module-finite over $k[\chi]$. In particular $\mathcal{Z}_\infty$ is Noetherian and ${\rm Ext}^*_\Lambda(\Bbbk, \Bbbk)$ is module-finite over $\mathcal{Z}_\infty$, so that ${\bf Fg}$ holds by Proposition \ref{Fgequivalence}. 
\end{proof}

\section{Further results}\label{sec:Further}

\subsection{Combinatorial characterisation of finite generation} 

Let $\Lambda$ be a monomial algebra. We have just established that $\Lambda$ satisfies ${\bf Fg}$ if and only if it is Gorenstein, and we would like to have a simple combinatorial criterion for testing the Gorenstein property of $\Lambda$ starting from the quiver and relations. 

For any $d$ with $1 \leq d < \infty$, we have seen in Theorem \ref{gorensteincharacterisation} a combinatorial characterisation for the condition that $\Lambda$ is Gorenstein of dimension $\dim \Lambda \leq d$ in terms of the structure of $(d-1)$-chains. More precisely, when this holds one sees a periodicity occuring in the structure $n$-chains for all $n \geq d$, in that all subsequent chains are obtained by appending the next path in a perfect cycle, which are periodic. 

Starting with an arbitrary monomial algebra $\Lambda$, this eventual periodicity of $n$-chains, if it occurs, could a priori begin in arbitrarily large degree $n \gg 0$, and so it is not clear that one can rule out the ${\bf Fg}$ property for $\Lambda$ by inspecting finitely many chains. 

As it turns out, there is a simple upper bound $n_\Lambda$ such that this periodicity either occurs for $n$-chains for all $n \geq n_\Lambda$ or never occurs at all. Recall that the (little) right finitistic dimension of $\Lambda$ is defined as 
\begin{align}
	{\rm rfindim}\ \! \Lambda := \sup \{ \pdim M \ | \ M \in \modsf \Lambda \textup{ and } \pdim M < \infty \}. 	
\end{align}
The left finitistic dimension is defined analogously via left modules, or simply via ${\rm lfindim}\ \! \Lambda = {\rm rfindim} \ \! \Lambda^{op}$. From the work of Igusa, Zacharia \cite{IZ} and Green, Kirkman and Kuzmanovich \cite{GKK}, one knows that the finitistic dimension of a monomial algebra is always finite. Moreover, if $\Lambda$ is any Gorenstein algebra then the finitistic dimension is also known to be finite, and in fact agrees with the Gorenstein dimension ${\rm rfindim}\ \! \Lambda = \dim \Lambda$ by \cite{Iw} (and then ${\rm lfindim} \ \! \Lambda = \dim \Lambda^{op} = \dim \Lambda = {\rm rfindim}\ \! \Lambda$). Letting $n_\Lambda$ be any upper bound for the finitistic dimension of a monomial algebra $\Lambda$, it follows that $\Lambda$ is either Gorenstein of dimension $\dim \Lambda \leq n_\Lambda$ or not Gorenstein at all. 

Simple upper bounds are known by work of Igusa, Zacharia \cite{IZ} and Zimmermann-Huisgen \cite{ZH}. For instance, let $K \geq 0$ be the minimal integer such that ${\sf r}^{K+1} = 0$ and set 
\begin{align}
	n_\Lambda := \min \{ \dim_k {\sf r},\ \! \dim_k \Lambda/{\sf r}^K \}. 
\end{align}
Then ${\rm rfindim}\ \! \Lambda \leq n_\Lambda$ by \cite[Section 3]{ZH}. Using this upper bound, we obtain the following decidable combinatorial criterion for ${\bf Fg}$: 

\begin{thm} Let $\Lambda$ be a monomial algebra. The following are equivalent:
	\begin{enumerate}[(1)] 
	\item $\Lambda$ satisfies ${\bf Fg}$.
	\item $\Lambda$ is Gorenstein.
	\item Let $n = {\rm rfindim}\ \! \Lambda$. Then every $n$-chain $\upgamma$ has a perfect path $t_\upgamma$ for tail. 
	\item Let $n = {\rm lfindim}\ \! \Lambda$. Then every $n$-chain $\upgamma$ has a perfect path $h_\upgamma$ for head. 
	\item Let $n = n_\Lambda$. Then every $n$-chain $\upgamma$ has a perfect path $t_\upgamma$ for tail. 
	\item Let $n = n_\Lambda$. Then every $n$-chain $\upgamma$ has a perfect path $h_\upgamma$ for head. 
\end{enumerate}
\end{thm}
\begin{proof} The equivalence of (1) and (2) follows from Theorem \ref{gorensteinfg}. Next, note that (3)-(4) and (5)-(6) are dual under $\Lambda \leftrightarrow \Lambda^{\rm op}$, and as (2) is self-dual it is enough to show that (3) is equivalent to (2) and (5) is equivalent to (2). 

	Set $N_\Lambda = {\rm rfindim}\ \! \Lambda$. Assuming (3), Theorem \ref{gorensteincharacterisation} shows that $\Lambda$ is Gorenstein of dimension $\dim \Lambda \leq N_\Lambda + 1$, in particular $\Lambda$ is Gorenstein and (2) holds. Conversely, by the remarks above $\Lambda$ is Gorenstein if and only if $\Lambda$ is Gorenstein of dimension $\dim \Lambda \leq N_\Lambda$. Assuming (2), we obtain that $\dim \Lambda < N_\Lambda + 1$ and so (3) follows from Theorem \ref{gorensteincharacterisation}. Finally the equivalence of (2) and (5) follows verbatim by setting $N_\Lambda = n_\Lambda$ instead in the previous argument. 
\end{proof}
\begin{rem} In practice the bound $n_\Lambda$ appears not to be very sharp and so periodicity may occur much faster; in any practical implementation better upper bounds should be used, see \cite{ZH} for further details. Our choice of bound $n_\Lambda$ was mainly for its simplicity to state. 
\end{rem}

\subsection{The structure of Hochschild cohomology} We conclude with some interesting corollaries on the structure of Hochschild cohomology. Let $\chi \in {\rm Ext}^p_\Lambda(\Bbbk, \Bbbk)$ be the $A_\infty$-central class defined in \eqref{chiclass}, and we also write $\chi \in {\rm HH}^p(\Lambda, \Lambda)$ for the corresponding lift. 

\begin{thm}\label{periodicHH} Let $\Lambda$ be a Gorenstein monomial algebra. Then ${\rm HH}^*(\Lambda, \Lambda)$ is eventually periodic; more precisely, taking cup product with $\chi$ gives an isomorphism 
	\[
		\chi \smile -: {\rm HH}^{n}(\Lambda, \Lambda) \xrightarrow{\cong} {\rm HH}^{n+p}(\Lambda, \Lambda) \quad \textup{ for all } n \geq \dim \Lambda + 1.
	\]
\end{thm} 
\begin{proof}
	We give a quantitative version of a standard filtration argument \cite[Prop. 1.4]{EHSST} . First, recall that for any $M, N \in \modsf \Lambda$ there are natural isomorphisms 
	\begin{align}
		{\rm HH}^*(\Lambda, {\rm Hom}_k(M, N)) \cong {\rm Ext}^*_\Lambda(M, N). 
	\end{align}
	Moreover one has ${\rm Hom}_k(\Bbbk, \Bbbk) \cong {\rm D}(\Bbbk) \otimes_k \Bbbk \cong \Lambda^{\sf e}/{\sf rad}\ \! \Lambda^{\sf e}$ for ${\rm D} = {\rm Hom}_k(-, k)$. For any bimodule $B \in \modsf \Lambda^{\sf e}$ the vector space ${\rm HH}^*(\Lambda, B)$ is a right-module over ${\rm HH}^*(\Lambda, \Lambda)$, and this module structure on ${\rm HH}^*(\Lambda, {\rm Hom}_k(\Bbbk, \Bbbk)) \cong {\rm Ext}^*_\Lambda(\Bbbk, \Bbbk)$ agrees with that given by the characteristic morphism $\varphi_\Bbbk$. Now, consider a Jordan-H\"older filtration of $\Lambda$ in $\modsf \Lambda^e$ 
	\begin{align}
		0 = B_{K+1} \subseteq B_K \subseteq \dots \subseteq B_2 \subseteq B_1 \subseteq B_0 = \Lambda	
	\end{align}
	with $B_i/B_{i+1} \in {\sf add}_{\Lambda^{e}}({\rm Hom}_k(\Bbbk, \Bbbk))$ for $0 \leq i \leq K$. The long exact sequence of ${\rm HH}^*(\Lambda, B)$ is compatible\footnote{This follows from binaturality of ${\rm Ext}^*_{\Lambda^e}(\Lambda, M)$ making it a right module over ${\rm Ext}^*_{\Lambda^e}(\Lambda, \Lambda)$; note that a priori this only makes the map $- \smile \chi$ commute with the long exact sequence, but we use $\chi \smile -$ for consistency with the theorem's claim and use graded-commutativity.} with the module structure over ${\rm HH}^*(\Lambda, \Lambda)$ and so gives rise to commutative diagrams of the form (writing ${\rm HH}^n(B) := {\rm HH}^n(\Lambda, B)$ for short)
	\[
		\xymatrix@C=12pt{ 
			{\rm HH}^{n-1}(B_i/B_{i+1}) \ar[d]^-{\chi \smile -}  \ar[r] & {\rm HH}^n(B_{i+1}) \ar[d]^-{\chi \smile - } \ar[r] & {\rm HH}^n(B_{i}) \ar[d]^-{\chi \smile - } \ar[r] & {\rm HH}^n(B_{i}/B_{i+1}) \ar[d]^-{\chi \smile - } \ar[r] & {\rm HH}^{n+1}(B_{i+1}) \ar[d]^-{\chi \smile - } \\ 
			{\rm HH}^{n-1}(B_i/B_{i+1}) \ar[r] & {\rm HH}^n(B_{i+1}) \ar[r] & {\rm HH}^n(B_{i}) \ar[r] & {\rm HH}^n(B_{i}/B_{i+1}) \ar[r] & {\rm HH}^{n+1}(B_{i+1}) \\ 
}
	\]
	Since $B_i/B_{i+1} \in {\sf add}_{\Lambda^e}({\rm Hom}_k(\Bbbk, \Bbbk))$, from the above paragraph and Prop. \ref{chiperiodicity} we see that the fourth column map is an isomorphism for $n \geq \dim \Lambda + 1$, and the first column map is an isomorphism for $n \geq \dim \Lambda + 2$ and an epimorphism for $n = \dim \Lambda + 1$. The $5$-Lemma then shows that if the second and fifth column maps are isomorphism for $n \geq \dim \Lambda + 1$, then so is the third column map. The result for $B_0 = \Lambda$ then follows by induction from the base case $B_K = B_K/B_{K+1} \in {\sf add}_{\Lambda^e}({\rm Hom}_k(\Bbbk, \Bbbk))$, which follows from Prop. \ref{chiperiodicity}. 
\end{proof}

Finally, recall that for $\Lambda$ Gorenstein the enveloping algebra $\Lambda^{\sf e} = \Lambda^{op} \otimes_k \Lambda$ is also Gorenstein, then of dimension $\dim \Lambda^{\sf e} = 2 \dim \Lambda$, see \cite[Prop. 6.1]{BIKP}. Following Buchweitz, we define the Tate-Hochschild cohomology as the Ext algebra of the bimodule $\Lambda \in {\rm D}_{sg}(\Lambda^{\sf e})$ in the singularity category of $\Lambda^{\sf e}$: 
\begin{align}
	\widehat{ {\rm HH}^*}(\Lambda, \Lambda) := {\rm Ext}^*_{ {\rm D}_{sg}(\Lambda^{\sf e})}(\Lambda, \Lambda). 
\end{align}
\begin{cor}\label{TateHH} Let $\Lambda$ be a Gorenstein monomial algebra. Then the Tate-Hochschild cohomology ring is given by periodic Hochschild cohomology: 
	\[
		\widehat{{\rm HH}^*}(\Lambda, \Lambda) \cong {\rm HH}^*(\Lambda, \Lambda)[\chi^{-1}]. 
	\] 
\end{cor} 
\begin{proof} We first find a good model for $\widehat{ {\rm HH}^*}(\Lambda, \Lambda)$ by constructing an appropriate complete resolution of the diagonal. 
	Let $P_* \xrightarrow{\sim}$ $ _\Lambda \Lambda_{\Lambda}$ be a minimal projective resolution of the diagonal bimodule $\Lambda$. We may represent the class $\chi \in {\rm HH}^p(\Lambda, \Lambda)$ by a chain map 
	\begin{align}\label{chirepresentative}
		\chi: P_{*+p} \to P_{*} 
	\end{align}
	The corresponding class $\varphi_\Bbbk(\chi) = \chi \in {\rm Ext}^p_\Lambda(\Bbbk, \Bbbk)$ is then represented by 
	\begin{align}
		\Bbbk \otimes_\Lambda \chi : \Bbbk \otimes_\Lambda P_{*+p} \to \Bbbk \otimes_\Lambda P_{*}. 
	\end{align}
	Minimality of $P_*$ means that $P_* \otimes_{\Lambda^e} \Lambda^e/{\sf rad}\ \! \Lambda^e \cong \Bbbk \otimes_\Lambda P_* \otimes_\Lambda \Bbbk$ has trivial differential, and so $\Bbbk \otimes_\Lambda P_*$ is a minimal right resolution of $\Bbbk$ over $\Lambda$. Right multiplication by $\chi$ then corresponds to 
\[
	\xymatrix{ {\rm Ext}^{n}_\Lambda(\Bbbk, \Bbbk) \ar@{=}[d] \ar[r]^-{- \cdot \chi} & {\rm Ext}^{n+p}_\Lambda(\Bbbk, \Bbbk) \ar@{=}[d] \\ 
		{\rm Hom}_\Lambda(\Bbbk \otimes_\Lambda P_{n}, \Bbbk) \ar@{=}[d] \ar[r]^-{ - \circ \ \! \chi} & {\rm Hom}_\Lambda(\Bbbk \otimes_\Lambda P_{n+p}, \Bbbk) \ar@{=}[d] \\
		{\rm Hom}_\Bbbk(\Bbbk \otimes_\Lambda P_{n} \otimes_\Lambda \Bbbk, \Bbbk) \ar[r]^-{ - \circ \ \! \chi} & {\rm Hom}_\Bbbk(\Bbbk \otimes_\Lambda P_{n+p} \otimes_\Lambda \Bbbk, \Bbbk). 
}
\]

The first row is an isomorphism in degree $n \geq \dim \Lambda + 1$, and the third row is $\Bbbk$-dual to the map $\chi$ tensored down 
\begin{align}
	\chi \otimes_{\Lambda^e} \Lambda^e/{\sf rad}\ \! \Lambda^e: P_{n+p} \otimes_{\Lambda^e} \Lambda^e/{\sf rad}\ \! \Lambda^e \to P_{n} \otimes_{\Lambda^e} \Lambda^e/{\sf rad}\ \! \Lambda^e 
\end{align}
which must then also be an isomorphism for $n \geq \dim \Lambda + 1$. By the Nakayama Lemma we conclude that $\chi: P_{n+p} \to P_n$ is an isomorphism in the same range. In particular $\Lambda$ has an eventually periodic minimal projective bimodule resolution. 

Let $C_*$ be the unique periodic, unbounded, acyclic complex of finite $\Lambda^e$ projectives which agrees with $P_*$ in degree $n \geq \dim \Lambda + 1$; the complex $C_*$ is then a complete resolution of (the Gorenstein-projective approximation of) $\Lambda$ over $\Lambda^e$. By Buchweitz's theorem \cite{Bu} we may compute the Tate-Hochschild cohomology algebra as 
\begin{align}
	\widehat{ {\rm HH}^*}(\Lambda, \Lambda) \cong {\rm H}^*({\rm Hom}_{\Lambda^e}(C_*, C_*))
\end{align}
The chain-map $\chi: P_{*+p} \to P_*$ extends by periodicity to $\chi: C_{*+p} \to C_*$, giving rise to a class $\chi \in \widehat{ {\rm HH}^p}(\Lambda, \Lambda)$ which is the image of $\chi \in {\rm HH}^p(\Lambda, \Lambda)$ under the natural algebra map $\psi: {\rm HH}^*(\Lambda, \Lambda) \to \widehat{ {\rm HH}^*}(\Lambda, \Lambda)$. More importantly, $\chi: C_{*+p} \to C_*$ represents the periodicity isomorphism by construction and therefore admits an inverse $\chi^{-1}: C_{*-p} \to C_*$, with $\chi^{-1} \in \widehat{ {\rm HH}^{-p}}(\Lambda, \Lambda)$. The natural map then factors through the localisation as $\psi = \widetilde{\psi} \circ u$ 
\begin{align}
	{\rm HH}^*(\Lambda, \Lambda) \xrightarrow{u} {\rm HH}^*(\Lambda, \Lambda)[\chi^{-1}] \xrightarrow{ \widetilde{\psi}} \widehat{ {\rm HH}^*}(\Lambda, \Lambda).	
\end{align}
By \cite[6.3.5]{Bu} the map $\psi: {\rm HH}^n(\Lambda, \Lambda) \to \widehat{ {\rm HH}^n}(\Lambda, \Lambda)$ is an isomorphism for all $n \geq \dim \Lambda^{\sf e} + 1 = 2 \dim \Lambda + 1$, and so $u$ is injective in the same degrees. Since multiplication by $\chi$ on ${\rm HH}^{\geq n}(\Lambda, \Lambda)$ acts by periodicity for $n \geq \dim \Lambda + 1$ by the last theorem, the map $u$ is surjective in degree $n \geq \dim \Lambda + 1$. Hence $u$, and therefore $\widetilde{\psi}$, is an isomorphism in degree $n \geq 2 \dim \Lambda + 1$. Finally since $\widetilde{\psi}$ commutes with the action of the periodicity operators $\chi^{\pm 1}$ we see that $\widetilde{\psi}$ is an isomorphism in all degrees.  
\end{proof}

\bibliographystyle{alpha} 
\bibliography{monomial.bib} 
\end{document}